\documentclass[10pt,a4paper]{amsart}
\usepackage[utf8]{inputenc}
\usepackage[english]{babel}
\usepackage{amsmath}
\usepackage{amsfonts}
\usepackage{amssymb}
\usepackage{comment}
\usepackage{enumitem}
\usepackage[hidelinks]{hyperref}
\usepackage{bm}

\usepackage{tikz-cd}

\numberwithin{equation}{section}

\newcounter{count}

\theoremstyle{plain}
\newtheorem{theorem}{Theorem}[section]
\newtheorem{lemma}[theorem]{Lemma}
\newtheorem{prop}[theorem]{Proposition}
\newtheorem{coro}[theorem]{Corollary}
\newtheorem{claim}[theorem]{Claim}

\theoremstyle{definition}
\newtheorem{defi}[theorem]{Definition}
\newtheorem{rema}[theorem]{Remark}
\newtheorem{example}[theorem]{Example}

\newcommand{\bR}{\mathbb{R}}
\newcommand{\bC}{\mathbb{C}}
\newcommand{\bQ}{\mathbb{Q}}
\newcommand{\bN}{\mathbb{N}}
\newcommand{\bZ}{\mathbb{Z}}
\newcommand{\bP}{\mathbb{P}}
\newcommand{\bK}{\mathbb{K}}

\newcommand{\cO}{\mathcal{O}}

\DeclareMathOperator{\Spec}{\mathrm{Spec}}

\DeclareMathOperator{\Rat}{\mathrm{Rat}}
\DeclareMathOperator{\ord}{\mathrm{ord}}

\DeclareMathOperator{\Div}{\mathrm{Div}}
\DeclareMathOperator{\divi}{\mathrm{div}}
\DeclareMathOperator{\Gal}{\mathrm{Gal}}

\DeclareMathOperator{\vol}{\mathrm{vol}}

\DeclareMathOperator{\Supp}{\mathrm{Supp}}

\author{François Ballaÿ}
\title[Arithmetic Okounkov bodies and positivity of adelic divisors]{Arithmetic Okounkov bodies and positivity of adelic Cartier divisors}
\date{\today}
\begin{document}

\begin{abstract}  In algebraic geometry, theorems of Küronya and Lozovanu characterize the ampleness and the nefness of a Cartier divisor on a projective variety in terms of the shapes of its associated Okounkov bodies. We prove the analogous result in the context of Arakelov geometry, showing that the arithmetic  ampleness and nefness of an adelic $\bR$-Cartier divisor $\overline{D}$ are determined by arithmetic Okounkov bodies in the sense of Boucksom and Chen. Our main results  generalize to arbitrary projective varieties criteria for the positivity of toric metrized $\bR$-divisors on toric varieties established by Burgos Gil, Moriwaki, Philippon and Sombra.  As an application, we show that the absolute minimum of $\overline{D}$ coincides with the infimum of the Boucksom--Chen concave transform, and we prove a converse to the arithmetic Hilbert-Samuel theorem under mild positivity assumptions.  We also establish new criteria for the existence of generic nets of small points and subvarieties.

\end{abstract}

\maketitle

\setcounter{tocdepth}{1}
\tableofcontents
\setcounter{tocdepth}{2}

\section{Introduction}

 The theory of Okounkov bodies, developed independently by Lazarsfeld and Musta\c{t}\u{a} \cite{LazMus} and by Kaveh and Khovanskii \cite{KavehKhovanskii}, builds on earlier work of Okounkov to establish a connection between algebraic geometry and convex geometry outside of the toric framework. Let $D$ be a big Cartier divisor on a smooth projective variety $X$  over an algebraically closed field. Given a valuation of maximal rank $\nu$ on the field of rational functions $\Rat(X)$, the Okounkov body of $D$ with respect to $\nu$ is the convex set $\Delta_\nu(D) \subseteq \bR^{\dim X}$ defined as 
  the closure of 
\[ \left\{\nu(s)/n\ \mid \ n\geq 1,  \ s \in \Gamma(X,nD)^\times  \right\}\]
for the Euclidean topology, where $\Gamma(X,nD)^\times$ denotes the space of non-zero global sections of $\cO_X(nD)$. This construction can be considered as a generalization of the rational polytope determined by a toric divisor on a smooth toric variety \cite[§~6.1]{LazMus}.
  Since their introduction, Okounkov bodies have attracted a lot of attention and it has been an active topic of research to describe which geometric properties and invariants of the pair $(X,D)$ they encode. It turned out that many aspects of the ``toric dictionary" describing the geometry of toric varieties in  combinatorial terms can be extended to all smooth projective varieties by using Okounkov bodies. Among the numerous efforts on this topic, Küronya and Lozovanu studied in depth the relations between the positivity of $D$ and its associated Okounkov bodies. They established remarkable characterizations for the nefness and the ampleness of $D$ in terms of the shapes of the sets $\Delta_\nu(D)$, when $\nu$ varies in suitable families of valuations \cite{KL,KLDoc}. Although  Küronya and Lozovanu worked in characteristic zero, these characterizations where recently generalized in arbitrary characteristic by Park and Shin \cite[Theorem 1.1]{ParkShin}.

Arakelov theory is a powerful approach to Diophantine geometry, that develops arithmetic analogues of tools from algebraic geometry to tackle deep problems in number theory. It allows to study the arithmetico-geometric properties of a projective variety $X$ defined over a number field (or more generally, a global field) by looking at its adelic Cartier divisors $\overline{D}$, which are usual Cartier divisors equipped with a suitable collection of metrics.
 In \cite{BPS14}, Burgos Gil, Philippon and Sombra started a program relating the arithmetic geometry of toric varieties to convex analysis, providing an arithmetic analogue of the toric dictionary in the context of Arakelov geometry. 
  In joint work with Moriwaki  \cite{BMPS}, they  established criteria based on convex analysis for the positivity of toric adelic $\bR$-Cartier divisors on toric varieties (such as arithmetic nefness and arithmetic ampleness in the sense of Zhang \cite{Zhangplav}).  Arithmetic analogues of the notion of Okounkov bodies were introduced independently by Yuan \cite{YuanOkounkov} and by Boucksom and Chen \cite{BC} in the context of Arakelov geometry\footnote{Although the constructions of Yuan and of Boucksom and Chen are closely related, they do not coincide in general (see \cite[§~4.3]{BC}). In this text, we focus on Boucksom and Chen's definition.}. Given an adelic Cartier divisor $\overline{D}$ on a projective variety $X$ over a number field $K$, arithmetic Okounkov bodies are known to encode the following fundamental invariants, under suitable mild assumptions: the arithmetic volume of $\overline{D}$ (\cite[Theorem A]{YuanOkounkov}, \cite[Theorem 2.8]{BC}), the height $h_{\overline{D}}(X)$ of $X$ with respect to $\overline{D}$ (\cite[Theorems 3.1 and 3.7]{BC}), and the essential minimum of the height function $\widehat{h}_{\overline{D}} \colon X(\overline{K}) \rightarrow \bR$ associated to $\overline{D}$ (\cite[Corollary 1.7]{BaEM}). In analogy with the geometric situation and in view of the results of Burgos Gil, Moriwaki, Philippon and Sombra in the toric case, it is natural to expect that arithmetic Okounkov bodies encode arithmetic ampleness and nefness of adelic Cartier divisors on non-necessarily toric projective varieties. The main goal of the present paper is to show that this is indeed the case. Our main results are analogues of the aforementioned theorems of Küronya and Lozovanu in the context of Arakelov geometry (see Theorems \ref{thmnefintro} and \ref{thmampleintro} below). In subsections \ref{paragapplicationHS} and \ref{paragapplicationgeneric}, we give some applications towards the arithmetic Hilbert-Samuel theorem and new criteria for the existence of generic nets of small points and subvarieties.

\subsection{Main results}

  Let $X$ be a normal and geometrically integral projective variety of dimension $d \geq 1$ over a global field $K$, and let $\overline{K}$ be an algebraic closure of $K$.  An adelic $\bR$-Cartier divisor on $X$ is a pair $\overline{D} = (D,(g_v)_{v\in \Sigma_K})$ consisting of an $\bR$-Cartier divisor and a suitable collection $(g_v)_{v\in \Sigma_K}$   of $D$-Green functions on the analytifications $X_{\bC_v}^{\mathrm{an}}$ of $X$, when $v$  varies among the places of $K$ (see Definition \ref{defiadelicCartier}). The notion of adelic $\bR$-Cartier divisors is due to Moriwaki \cite{MoriwakiMAMS}, and generalizes the one of adelic line bundles in the sense of Zhang \cite{Zhangadelic}. Let $\nu$ be a valuation of maximal rank on $\Rat(X_{\overline{K}})$ (see Definition \ref{defivaluation}). 
 Assuming that $D$ is big, Boucksom and Chen \cite{BC} defined a concave function 
 \[G_{\overline{D}, \nu} \colon \Delta_{\nu}(D) \rightarrow \bR \cup \{-\infty\}\]
called the \emph{concave transform} of $\overline{D}$ with respect to $\nu$, which encodes important information on the arithmetic of the graded linear series $\bigoplus_{n\in\bN} \Gamma(X,nD)$  (see §~\ref{paragconcave}). Following \cite[Definition 2.7]{BC}, we define the \emph{arithmetic Okounkov body} of $\overline{D}$ with respect to $\nu$ as 
 \[\widehat{\Delta}_{\nu}(\overline{D})  = \{ (\alpha, t) \in \Delta_{\nu}(D)\times \bR_+ \ | \ G_{\overline{D},\nu}(\alpha) \geq t  \}.\]
 It is a compact convex subset of $\bR^{d+1}$.  In subsection \ref{paragarithmNO}, we will extend the definition of $\widehat{\Delta}_{\nu}(\overline{D})$  to include the case when $D$ is not big, by using limits of arithmetic Okounkov bodies in the above sense.

 A key notion to study the positivity of an adelic $\bR$-Cartier divisor $\overline{D} \in \widehat{\Div}(X)_{\bR}$ is its associated height function $\widehat{h}_{\overline{D}} \colon X({\overline{K}}) \rightarrow \bR$, which measures the arithmetic complexity of closed points on $X$ and also plays a fundamental role in applications of Arakelov theory to Diophantine geometry. The central result of this paper gives a lower bound for the height of points in terms of Boucksom and Chen's concave transform.
 \begin{theorem}[Theorem \ref{thmkey}]\label{thmkeyintro} Let $x \in X_{\overline{K}}$ be a closed point, and let $\nu$ be a valuation of maximal rank on $\Rat(X_{\overline{K}})$  centered at $x$. For any adelic $\bR$-Cartier divisor $\overline{D} = (D,(g_v)_{v \in \Sigma_K})$ such that $D$ is nef, we have $0_{\bR^d} \in \Delta_\nu(D)$ and 
 \[\widehat{h}_{\overline{D}}(x) \geq G_{\overline{D},\nu}(0_{\bR^d}).\]
 \end{theorem}
 Roughly speaking, at least when $D$ is big one can think of the quantity $G_{\overline{D}, \nu}(0_{\bR^d})$ as a measure of the arithmetic ``size" of global sections of multiples of $D$ with arbitrarily small valuation. The key new ingredient in the proof of Theorem \ref{thmkeyintro} is a lower bound for the height of $x$ in terms of the order of vanishing of a small section at $x$, that we shall outline at the end of this introduction (§~\ref{paragideaproof}). As an application of Theorem \ref{thmkeyintro}, we obtain new expressions for the absolute minimum of $\overline{D}$, defined as $\zeta_{\mathrm{abs}}(\overline{D}) = \inf_{x\in X(\overline{K})}\widehat{h}_{\overline{D}}(x)$. We denote by  $\mathcal{V}(X_{\overline{K}})$ the set of valuations of maximal rank on $\Rat(X_{\overline{K}})$.
\begin{coro}[Corollary \ref{corominabsG}]\label{corominabsintro} Let $\overline{D} = (D, (g_v)_{v\in \Sigma_K}) \in \widehat{\Div}(X)_\bR$ be a semi-positive adelic $\bR$-Cartier divisor. Then we have 
\[ \zeta_{\mathrm{abs}}(\overline{D}) = \inf_{\nu \in \mathcal{V}(X_{\overline{K}})} G_{\overline{D},\nu}(0_{\bR^d}), \]
where the infimum is over all valuations of maximal rank on $\Rat(X_{\overline{K}})$.
 Moreover, if $D$ is big then 
\[\zeta_{\mathrm{abs}}(\overline{D}) = \inf_{\alpha \in \Delta_\nu(D)} G_{\overline{D},\nu}(\alpha)\]
for any $\nu \in \mathcal{V}(X_{\overline{K}})$.
\end{coro} 
 The lower bounds for $\zeta_{\mathrm{abs}}(\overline{D})$ given in Corollary \ref{corominabsintro} are straightforward consequences of Theorem \ref{thmkeyintro}, together with the fact that the infimum of the concave transform does not depend on the choice of the valuation when $D$ is big (see Lemma \ref{lemmainfG}). The upper bounds follow from Zhang's arithmetic Nakai-Moishezon theorem (Theorem \ref{thmNM}).  We mention that Corollary \ref{corominabsintro} remains valid when $\mathcal{V}(X_{\overline{K}})$ is replaced by any subset of valuations of maximal rank whose centers cover $X_{\overline{K}}$ (see Corollary \ref{corominabsG} for a precise statement). 
 The second equality of Corollary \ref{corominabsintro} is a counterpart to \cite[Corollary 1.3]{BaEM}, where it was shown that the maximum of $G_{\overline{D},\nu}$ coincides with the essential minimum of the height function $\widehat{h}_{\overline{D}}$. These results give an affirmative answer to a question of Burgos Gil, Philippon and Sombra \cite[Remark 3.15]{BPSmin}, and open new approaches to study the absolute and essential minima through convex analysis.

We now present straightforward consequences of Theorem \ref{thmkeyintro} and Corollary \ref{corominabsintro} that relate arithmetic positivity to arithmetic Okounkov bodies. Following \cite{MoriwakiMAMS, BMPS}, we say that $\overline{D}$ is nef if it is semi-positive and if $\zeta_{\mathrm{abs}}(\overline{D})\geq 0$. The following  result is a counterpart to \cite[Theorem A]{KL} in the context of Arakelov geometry. 

\begin{theorem}[Corollary \ref{coronef}]\label{thmnefintro} Let $\overline{D} \in \widehat{\Div}(X)_\bR$ be a semi-positive adelic $\bR$-Cartier divisor. The following assertions are equivalent:
\begin{enumerate}
\item $\overline{D}$ is nef;
\item for every $\nu \in \mathcal{V}(X_{\overline{K}})$, we have $0_{\bR^{d+1}} \in \widehat{\Delta}_{\nu}(\overline{D})$;
\item for every closed point $x \in X_{\overline{K}}$, there exists a valuation of maximal rank $\nu \in \mathcal{V}(X_{\overline{K}})$ centered at $x \in X_{\overline{K}}$ such that $0_{\bR^{d+1}} \in \widehat{\Delta}_{\nu}(\overline{D})$.
\end{enumerate}
\end{theorem}

  We now turn to characterizations of arithmetic ampleness (see Definition \ref{defipositivity}). In view of the arithmetic Nakai-Moishezon theorem of Zhang (see Theorem \ref{thmNM}), it follows from Corollary \ref{corominabsintro} that a semi-positive adelic $\bR$-Cartier divisor $\overline{D} = (D,(g_v)_{v\in\Sigma_K})$ is ample if and only if the underlying divisor $D$ is (geometrically) ample and the concave transform $G_{\overline{D},\nu}$ has a positive lower bound for some choice of $\nu \in \mathcal{V}(X_{\overline{K}})$ (see Corollary \ref{coroample}). By combining this result with characterizations of geometric ampleness due to  Küronya and Lozovanu and to Park and Shin \cite{KL, KLDoc, ParkShin},  we shall prove more intrinsic statements relating arithmetic ampleness to the shapes of arithmetic Okounkov bodies. For example, we obtain the following theorem when we focus on   arithmetic Okounkov bodies $\widehat{\Delta}_{\widetilde{Y}_{\bullet}}(\overline{D})$ defined with respect to valuations arising from infinitesimal flags $\widetilde{Y}_\bullet$ (Example  \ref{examplevalinf}).  Given a real number $\lambda \geq 0$, we denote by  $ \Delta_{\lambda}^{-1}$   the inverted standard simplex of size $\lambda \geq 0$  (see \cite{ParkShin} and § \ref{paraginfinitesimal}).
   \begin{theorem}[Corollary \ref{thminfample}]\label{thmampleintro} Assume that $X$ is smooth, and let $\overline{D} \in \widehat{\Div}(X)_\bR$ be a semi-positive adelic $\bR$-Cartier divisor. The following assertions are equivalent: 
\begin{enumerate}
\item $\overline{D}$ is ample;
\item there exists a real number $\lambda > 0$ such that $\Delta_{\lambda}^{-1} \times \{\lambda\} \subseteq \widehat{\Delta}_{\widetilde{Y}_\bullet}(\overline{D})$
for every infinitesimal flag  $\widetilde{Y}_{\bullet}$ on $X_{\overline{K}}$;
\item there exists a real number $\lambda > 0$ such that for every closed point $x \in X_{\overline{K}}$, there exists an infinitesimal flag $\widetilde{Y}_{\bullet}$ over $x$   with $\Delta_{\lambda}^{-1} \times \{\lambda\} \subseteq \widehat{\Delta}_{\widetilde{Y}_{\bullet}}(\overline{D})$.
\end{enumerate}
\end{theorem}
This statement is an analogue of \cite[Theorem B]{KL} in the context of Arakelov geometry. We shall also prove a  similar result  concerning admissible flags (Corollary \ref{thmadmample}), that provides an arithmetic analogue to \cite[Corollary 3.2]{KLDoc}.
 In the particular case where the pair $(X,\overline{D})$ is toric and assuming that $\overline{D}$ is semi-positive, one can recover the characterizations of arithmetic ampleness and nefness of \cite[Theorem 2]{BMPS} from Theorems \ref{thmnefintro} and \ref{thmampleintro} (see Remark \ref{rematoric}).
 
\subsection{Application to the arithmetic Hilbert-Samuel theorem}\label{paragapplicationHS}

 Given an adelic $\bR$-Cartier divisor $\overline{D} = (D,(g_v)_{v\in\Sigma_K}) \in \widehat{\Div}(X)_\bR$, we denote by 
 \[\widehat{\Gamma}(X,\overline{D}) = \{s \in \Gamma(X,D) \ | \ \|s\|_{v,\sup} \leq 1 \ \forall \  v \in \Sigma_K\}\] the set of small sections of $\overline{D}$  (see §~\ref{paragvolumes}). We let $\widehat{h}^0(X, \overline{D}) = \ln \# \widehat{\Gamma}(X, \overline{D})$ if $K$ is a number field, and $\widehat{h}^0(X, \overline{D}) = \dim_k \widehat{\Gamma}(X, \overline{D})$ if $K$ is the function field of a regular projective curve over a field $k$.
  The arithmetic Hilbert-Samuel theorem asserts  that if $\overline{D}$ is nef, then  for any $\overline{N} \in \widehat{\Div}(X)_\bR$ we have
\begin{equation}\label{eqHS}
\widehat{h}^0(X,n\overline{D} + \overline{N})  = \frac{h_{\overline{D}}(X)}{(d+1)!}n^{d+1} + o(n^{d+1})
\end{equation}
 when $n \gg 1$, where $h_{\overline{D}}(X)$ denotes the height of $X$ with respect to $\overline{D}$ (see §~\ref{sectionheightsemipos}). The identity \eqref{eqHS} can be reformulated as $\widehat{\vol}(\overline{D}) = h_{\overline{D}}(X)$, where $\widehat{\vol}(\overline{D})$ is the arithmetic volume of $\overline{D}$ (see §~\ref{paragvolumes} and Remark \ref{remacoroHS}). When $K$ is a number field, the arithmetic Hilbert-Samuel theorem was originally proved by Gillet and Soulé as a consequence of the arithmetic Riemann-Roch theorem,  under the hypothesis that $\overline{D}$ is an ample adelic Cartier divisor and that the metrics at the non-Archimedean places are given by a global projective model over the ring of integers of $K$ (see \cite[Theorems 8 and 9]{GS92}). In the level of generality considered above, it was proved by Moriwaki  by using the continuity of arithmetic volumes (see \cite[Theorem A]{MoriwakiContVol}, \cite[Theorem 5.3.2]{MoriwakiMAMS} and Theorem \ref{thmHodge} \emph{infra}). As a consequence of Theorem \ref{thmkeyintro}, we shall establish a converse to the arithmetic Hilbert-Samuel theorem, proving that \eqref{eqHS} is actually a criterion for nefness when $\overline{D}$ is semi-positive and $D$ is big.
\begin{theorem}[Corollary \ref{coronefHodge}]\label{coroHSintro}  Let $\overline{D}= (D,(g_v)_{v \in \Sigma_K}) \in \widehat{\Div}(X)_\bR$ be semi-positive. If $D$ is big, the following conditions are equivalent:
\begin{enumerate}[leftmargin=0.8cm]
\item  $\overline{D}$ is nef;
\item $\widehat{\vol}(\overline{D}) = h_{\overline{D}}(X)$;
\item for any $\overline{N} \in \widehat{\Div}(X)_\bR$, we have
\[\widehat{h}^0(X,n\overline{D} + \overline{N}) = \frac{h_{\overline{D}}(X)}{(d+1)!}n^{d+1} + o(n^{d+1}).\]
\end{enumerate}
\end{theorem}
To our knowledge, this result was only known under the additional assumption that $d= 1$ \cite[Theorem 7.3.1]{MoriwakiMAMS} or that $(X,\overline{D})$ is toric \cite[Corollary 6.2]{BMPS} up to now. Theorem \ref{coroHSintro} is a direct consequence of Corollary \ref{corominabsintro} thanks to fundamental formulae due to Boucksom and Chen relating arithmetic volumes to the concave transform (see Theorem \ref{thmBC}).

\subsection{Application to generic nets of small points and subvarieties}\label{paragapplicationgeneric}  Let $\overline{D}= (D,(g_v)_{v \in \Sigma_K}) \in \widehat{\Div}(X)_\bR$ be a semi-positive $\bR$-Cartier divisor on $X$ such that $D$ is big. By a theorem of Zhang \cite[Theorem 5.2]{Zhangplav} (see Theorem \ref{thmZhang} \textit{infra}), we have 
\begin{equation}\label{ineqsmallnet}
\lim_m \widehat{h}_{\overline{D}}(p_m) \geq \widehat{h}_{\overline{D}}(X) := \frac{h_{\overline{D}}(X)}{(d+1)\deg_D(X)}
\end{equation}
for any  generic net of points $(p_m)_{m}$  in $X_{\overline{K}}$ (see Definition \ref{defigeneric}). In their seminal article \cite{SUZ}, Spziro, Ullmo and Zhang pioneered the study of the limit distribution of Galois orbits of a generic net of points for which equality occurs in \eqref{ineqsmallnet}. Their work has proven to have remarkable applications towards the Bogolomov conjecture, and has been widely generalized by many authors. We refer the reader to \cite[Introduction]{BPRS} for a more detailed introduction to this topic.  These results lead naturally to the following question: under which conditions does there exist a generic net of points $(p_m)_{m}$  in $X_{\overline{K}}$ such that equality holds in \eqref{ineqsmallnet}? In \cite[Theorem 1.5]{BaEM}, it was shown that such a net exists if and only if $G_{\overline{D}, \nu}$ is constant for any valuation of maximal rank $\nu \in \mathcal{V}(X_{\overline{K}})$. Combining this result with Corollary \ref{corominabsintro}, we shall prove the following theorem.

\begin{theorem}[Theorem \ref{thmgeneric}]\label{thmgenericintro}
 Let $\overline{D} = (D,(g_v)_{v\in \Sigma_K}) \in \widehat{\Div}(X)_\bR$ be semi-positive, with $D$ big. The following conditions are equivalent:
 \begin{enumerate}[leftmargin=0.8cm]
 \item there exists a generic net of points $(p_m)_{m}$ in $X_{\overline{K}}$ such that 
 \[\lim_m \widehat{h}_{\overline{D}}(p_m) = \widehat{h}_{\overline{D}}(X);\]
 \item there exists a generic net of subvarieties $(Y_m)_{m}$ in $X_{\overline{K}}$ such that 
 \[\lim_m \widehat{h}_{\overline{D}}(Y_m) = \zeta_{\mathrm{abs}}(\overline{D});\]
 \item we have $\widehat{h}_{\overline{D}}(x) \geq \widehat{h}_{\overline{D}}(X)$ for every closed point $x \in X_{\overline{K}}$;
 \item we have $\widehat{h}_{\overline{D}}(Y) \geq \widehat{h}_{\overline{D}}(X)$ for every subvariety $Y \subseteq X_{\overline{K}}$ with $\deg_D(Y)>0$.

 \end{enumerate}
\end{theorem}
 
 It is worth to note that this statement shows in particular that Yuan's theorems on equidistribution of Galois orbits of small points and subvarieties \cite[Theorem 3.1 and Theorem 3.6]{Yuan} actually share the same assumption. We recall them below in a single concise statement.
 \begin{theorem}[Yuan] Assume that $K$ is a number field. Let $\overline{D}=(D,(g_v)_{v\in\Sigma_K})\in \widehat{\Div}(X)$ be a semi-positive adelic Cartier divisor on $X$ with $D$ ample. For any generic net of subvarieties $(Y_m)_{m}$ in $X_{\overline{K}}$ such that 
 \[\lim_m \widehat{h}_{\overline{D}}(Y_m) = \zeta_{\mathrm{abs}}(\overline{D})\]
 and for any place $v$ of $K$, the Galois orbits of $(Y_m)_{m}$ are equidistributed in the analytic space $X_{\bC_v}^{\mathrm{an}}$ with respect to the measure $c_{1,v}(\overline{D})^{d}/\deg_D(X)$.
 \end{theorem}
 
In the above statement, the notion of equidistribution of Galois orbits is to be understood in the sense of \cite[§~3]{Yuan}.
 
\subsection{Ideas of proof}\label{paragideaproof} Let $\overline{D} \in \widehat{\Div}(X)_\bR$ be an adelic $\bR$-Cartier divisor. 
 The main new ingredient in our proof of Theorem \ref{thmkeyintro} is the following result controlling the height of a point $x$ in terms of the order of vanishing $\ord_x s$ at $x$ of small section $s \in \widehat{\Gamma}(X,\overline{D})$. It can be interpreted as an adelic version of  Cauchy's inequality for the global sections of $\overline{D}$.
\begin{prop}[Proposition \ref{propCauchy}]\label{propCauchyintro} Assume that $D \in \Div(X)$ is an ample Cartier divisor. Let $x \in X_{\overline{K}}$ be a regular closed point and $\varepsilon > 0$ a real number. There exists $\rho(\overline{D},x,\varepsilon) \in \bR$ such that
\begin{equation}
 \widehat{h}_{\overline{D}}(x) + \varepsilon \geq - \rho(\overline{D},x,\varepsilon)\frac{\ord_x s}{m}
\end{equation}
 for any positive integer $m$ and for any non-zero small section $s \in \widehat{\Gamma}(X, m\overline{D})$.
\end{prop}
In order to sketch the proof of this proposition, we assume that $x \in X(K)$ is a $K$-rational point for simplicity. Given a small section  $s \in \widehat{\Gamma}(X, m\overline{D})$, we construct a local section $\bm{\partial}s \in \cO_X(mD)$ not vanishing at $x$ by applying to $s$ a differential operator of order $\ord_x s$. For every place $v \in \Sigma_K$, we apply Cauchy's inequality in severable variables on a suitable polydisc of radius $\rho_{v} > 0$ in $X_{\bC_v}$ to prove an inequality of the form  
 \begin{equation}\label{ineqCauchyintro}
 \|\bm{\partial}s(x)\|_v \leq \alpha_v(\rho_v)^{m}\rho_v^{-\ord_x s}\|s\|_{v, \sup} \leq \alpha_v(\rho_v)^{m}\rho_v^{-\ord_x s},
 \end{equation}
 where $\alpha_v(\rho_v)$ is an error term controlling the behaviour of a trivialization of $\cO_X(D)$ on the polydisc.  Using that everything is defined over $K$, we show that there exists a finite set $\mathcal{P}$ of places (depending only on $\overline{D}$ and $x$) such that for every $v \in \Sigma_K \setminus \mathcal{P}$, we can choose $\rho_v = 1$ and ensure that $ \alpha_v(\rho_v)= 1$ (see Claims \ref{claimtriv} and \ref{claimsection}). By definition, we have 
\[ m\widehat{h}_{\overline{D}}(x) = -\sum_{v\in \Sigma_K} n_v(K) \ln\|\bm{\partial}s(x)\|_v\]
where $n_v(K)$ is the local weight of $K$ at $v$ (see §~\ref{paragweights}). Therefore summing up the inequalities \eqref{ineqCauchyintro} gives the lower bound
 \[ \widehat{h}_{\overline{D}}(x) + \sum_{v \in \mathcal{P}} n_v(K)\ln (\alpha_v(\rho_v)) \geq  - \frac{\ord_xs}{m} \times \sum_{v \in \mathcal{P}} n_v(K) \ln (1/\rho_v).\]
Proposition \ref{propCauchyintro} follows by choosing $\rho_v$ sufficiently small for every $v \in \mathcal{P}$, and by using the fact that  $\lim_{\rho_v \rightarrow 0} \ln(\alpha_v(\rho_v))=0$ (which is a direct consequence of  the continuity of the metrics).

  We now briefly explain how Proposition \ref{propCauchyintro} implies Theorem \ref{thmkeyintro} in the case where $X$ is smooth and $D \in \Div(X)$ is an ample Cartier divisor.  The general case follows by using de Jong's alteration theorem and continuity arguments (see Lemma \ref{lemmaConcalt}). Let $x \in X_{\overline{K}}$ be a closed point and let $\nu \in \mathcal{V}(X_{\overline{K}})$ be a valuation of maximal rank centered at $x$.  The point is to observe that if $G_{\overline{D},\nu}(0_{\bR^d}) > 0$, then  for any $\sigma > 0$, there exist an integer $m > 0$ and a small section $s \in \widehat{\Gamma}(X,m\overline{D})$ such that $\ord_x(s) < m\sigma$ (see Claim \ref{claimexists0}).  
    Applying Proposition  \ref{propCauchyintro} and letting $\sigma$ tend to zero shows that $\widehat{h}_{\overline{D}}(x) + \varepsilon \geq 0$ for every $\varepsilon > 0$. This gives the implication
    \[G_{\overline{D},\nu}(0_{\bR^d}) > 0 \Longrightarrow \widehat{h}_{\overline{D}}(x) > 0,\]
 and the inequality of Theorem \ref{thmkeyintro} follows by rescaling metrics.

 \subsection{Organization of the paper} 
  In section \ref{sectionadelicdiv} we recall the definition and basic properties of adelic $\bR$-Cartier divisors. 
   Section \ref{sectionCauchy} contains the main new technical ingredient of this paper, namely the proof of  Proposition \ref{propCauchyintro} (see Proposition \ref{propCauchy}). In section \ref{sectionBCarithmNO} we recall the definition of the Boucksom--Chen concave transform and we define arithmetic Okounkov bodies; the only new material in this section is a slight generalization of Boucksom and Chen's construction, for which we do not require the underlying divisor $D$ to be big (§§~\ref{paraggeneralconc} and \ref{paragarithmNO}). In section \ref{sectionpositivNO}, we prove Theorem \ref{thmnefintro}, Corollary \ref{corominabsintro}, and Theorem \ref{thmampleintro}. Finally, we present applications to the arithmetic Hilbert-Samuel theorem and to generic nets of small points and subvarieties in section \ref{sectionapplication}, where we prove Theorems \ref{coroHSintro} and \ref{thmgenericintro}.
 
 \subsection{Conventions and terminology}\label{sectionconv}

\subsubsection{} A scheme is integral if it is reduced and irreducible. A variety over a field $K$ is an integral scheme of finite type on $\Spec K$. Given a normal variety $X$ over $K$, we denote by $\Div(X)$ the group of Cartier divisors on $X$, and by $\Rat(X)$ the function field of $X$. If $\bK$ denotes $\bZ$, $\bQ$ or $\bR$, we let $\Div(X)_{\bK} = \Div(X) \otimes_\bZ \bK$ and $\Rat(X)_{\bK}^\times = \Rat(X)^\times \otimes_\bZ \bK$. For any $D \in \Div(X)_\bR$, we  define the set of \emph{non-zero rational $\bK$-sections} of $D$  as
\[\Rat(X,D)_\bK^\times := \Rat(X)_\bK^\times \times \{D\},\]
and the set of  \emph{non-zero global $\bK$-sections} of $D$ as 
\[\Gamma(X, D)_\bK^\times = \{ (\phi, D) \in \Rat(X,D)_\bK^\times \ | \ (\phi) + D \geq 0\}.\]
  When $\bK = \bZ$, we shall omit the subscript $\bZ$ in the above notations. The \emph{support} $\Supp(D)$ of $D \in \Div(X)_\bR$ is defined as the support of the $\bR$-Weil divisor associated to $D$. For any field extension $K'$ of $K$, we let $X_{K'} = X \times_K K'$ and we denote by $D_{K'}$ the pullback of $D$ to $X_{K'}$. A subvariety $Y$ of $X_{K'}$ is an irreducible Zariski-closed subset of $X_{K'}$ equipped with the reduced induced scheme structure. For any $D \in \Div(X)_{\bR}$, the \emph{degree} of a subvariety $Y \subseteq X_{\overline{K}}$ with respect to $D$ is defined as $\deg_D(Y) =  D_{\overline{K}}^{\dim Y} \cdot Y$, where $\overline{K}$ is an algebraic closure of $K$. 
  
\subsubsection{} Let $S$ be a Noetherian integral scheme with function field $\Rat(S)$ and let $X$ be a projective variety over $\Rat(S)$. Given a dense open subset $U \subseteq S$, a \emph{model of $X$} over $U$ is an integral scheme equipped with a projective morphism $\pi\colon \mathcal{X} \rightarrow U$ such that $X = \mathcal{X}\times_U \Rat(S)$. Let $D \in \Div(X)_\bR$,  that is $D = \sum_{i = 1}^n a_i D_i$ for some Cartier divisors $D_1,\ldots, D_n$ on $X$ and $a_1, \ldots, a_n \in \bR$. We assume that for each $i \in \{1, \ldots, n\}$ there exists a Cartier divisor $\mathcal{D}_i$ on $\mathcal{X}$ such that $\mathcal{D}_i \cap X =D_i$, and we let $\mathcal{D} = \sum_{i = 1}^n a_i \mathcal{D}_i$. The pair $(\mathcal{X}, \mathcal{D})$ is called a \emph{model of $(X,D)$} over $U$.  We say that the model 
$(\mathcal{X}, \mathcal{D})$ is relatively nef if $\mathcal{D}$ is relatively nef with respect to $\pi$ (see \cite[0.5.6]{MoriwakiMAMS}).
\subsubsection{} Let $K_0$ denote either $\bQ$ or the field of functions $k(T)$, where $k$ is an arbitrary field. By definition, a global field $K$ is a finite extension of $K_0$. We denote by $\Sigma_K$ the set of places of $K$, and by $\Sigma_{K,\infty} \subseteq \Sigma_K$ the set of Archimedean places (which is empty if $K_0=k(T)$). For any $v \in \Sigma_K$, we let $K_v$ be the completion of $K$ with respect to $v$ and $K_{0,v}$ be the completion of $K_0$ with respect to the restriction of $v$ to $K_0$. We fix an algebraic closure of $K_v$, and we denote by $\bC_v$ its completion.

\subsubsection{}\label{paragweights} We shall normalize absolute values on a global field $K$. Assume first that $K$ is a number field. For each place $v \in \Sigma_K$, we let $|.|_v$ be the unique absolute value on $K$ satisfying the following property~: $|p|_v = p^{-1}$ if $v$ is a finite place over a prime number $p$, and $|.|_v = |.|$ is the restriction of the usual absolute value on ${\mathbb{C}}$ if $v$ is Archimedean. For each $v \in \Sigma_K$,  we let $n_v(K) = [K_v:{\mathbb{Q}}_v]/[K:{\mathbb{Q}}]$.  We now assume that $K$ is a finite extension of $K_0 = k(T)$ for some field $k$. Equivalently, $K = \Rat(C_K)$ is the function field of a regular projective curve $C_K$ over $k$ equipped with a finite morphism $\varphi_K \colon C_K \rightarrow \mathbb{P}^1_k$, unique up to $k$-isomorphism.  The set of places of $K$ is in one-to-one correspondence with the set  of closed points of $C_K$. For each $v \in \Sigma_K$ and each $f \in K^\times =  {\Rat}(C_K)^\times$, we denote by $\ord_v(f)$ the order of $f$ in the discrete valuation ring ${\mathcal{O}}_{C_K,v}$ { and by $e_v(\varphi_K)$ the ramification index of $\varphi_K\colon C_K \rightarrow {\mathbb{P}}^1_k$ at $v \in C_K$. We consider the absolute value $|.|_v$ on $K$ given by $|f|_v = e^{-\ord_v(f)/e_v(\varphi_K)}$ and we let 
   \[n_v(K) = \frac{e_v(\varphi_K) [\kappa(v):k]}{[K:K_0]},\]
 where $\kappa(v)$ denotes the residue field of $v$ in $C_K$}.

\subsubsection{} Let $K$ be a global field and let $X$ be a scheme of finite type on $\Spec K$. For any place $v \in \Sigma_K$ and for any subscheme $U \subseteq X_{\bC_v} = X \times_K \Spec \bC_v$, we denote by $U^{\mathrm{an}}$ the analytification of $U$ in the sense of Berkovich \cite{Ber90} (see \cite[§~1.2]{BPS14} and \cite[§~1.3]{MoriwakiMAMS} for a short introduction). Note that $U(\bC_v)$ is a dense subset of $U^{\mathrm{an}}$. If $U = \Spec A$ is affine, then the underlying set of $U^{\mathrm{an}}$ is the set of multiplicative semi-norms on $A$ extending $|.|_v$. In that case, we write $|a(z)|_v=|a|_z $ for $a \in A$ and $z = |.|_z \in U^{\mathrm{an}}$.
  
\subsubsection{}  Throughout this paper,  $X$ denotes a normal and geometrically integral projective variety of dimension $d \geq 1$  over a global field $K$. We fix an algebraic closure $\overline{K}$ of $K$, and we let $\mathcal{S}$ be the scheme defined as follows:
 \begin{itemize}
 \item if $K$ is a number field with ring of integers $\cO_K$, $\mathcal{S} = \Spec \cO_K$;
 \item if $K$ is a finite extension of $k(T)$ for some field $k$, $\mathcal{S} = C_K$ is a regular projective curve over $k$ such that $\Rat(C_K) = K$.
 \end{itemize}
\section{Adelic \texorpdfstring{$\bR$}{R}-Cartier divisors}\label{sectionadelicdiv}

In this section we define adelic $\bR$-Cartier divisors and we recall various notions of positivity. We mainly follow \cite{MoriwakiMAMS} and \cite{BMPS}.

\subsection{Definitions}

 Let $D \in \Div(X)_{\mathbb{R}}$ and  $v \in \Sigma_K$. We  consider an open covering $X_{\bC_v} = \cup_{i=1}^{\ell}U_i$  such that $D_{\bC_v}$ is defined by $f_i \in { { \Rat}(X_{\bC_v})^\times_{\mathbb{R}}}$ on $U_i$ for each $i \in \{1,\ldots, \ell\}$. A continuous (respectively smooth) $D$\textit{-Green function} on $X_{\bC_v}^{\mathrm{an}}$ is a  function 
\[g_v \colon X_{\bC_v}^{\mathrm{an}}\setminus (\Supp D_{\bC_v})^{\mathrm{an}} \rightarrow {\mathbb{R}}\]
such that $g_v + \ln |f_i|^2_v$ extends to a continuous (respectively smooth) function on the analytification $U_i^{\mathrm{an}}$ of $U_i$ for each $i \in \{1, \ldots, \ell\}$. We refer the reader to \cite[sections 1.4 and 2.1]{MoriwakiMAMS} for more detail on Green functions. If $(\mathcal{X}, \mathcal{D})$ is a model of $(X,D)$ over a dense open subset $U \subseteq \mathcal{S}$, then for each $v \in U$ we denote by $g_{\mathcal{D},v}$ the $D$-Green function on $X_{\bC_v}^{\mathrm{an}}$ induced by $\mathcal{D}$ (see \cite[§~0.2]{MoriwakiMAMS}).  

\begin{defi}\label{defiadelicCartier} An adelic $\bR$-Cartier divisor on $X$ is the data $\overline{D} = (D,(g_v)_{v \in \Sigma_K})$ of an $\bR$-Cartier divisor $D \in \Div(X)_\bR$ and, for each place $v \in \Sigma_K$, of a continuous $D$-Green function $g_v \colon X_{\bC_v}^{\mathrm{an}} \rightarrow \bR$ such that the following conditions are satisfied:
\begin{enumerate}
\item for each $v \in \Sigma_K$, $g_v$ in invariant under the action of $\Gal(\bC_v/K_v)$;
\item there exists a model $(\mathcal{X}, \mathcal{D})$ of $(X,D)$ over a dense open subset $U \subseteq \mathcal{S}$ such that $g_v = g_{\mathcal{D},v}$ for every $v \in U$.
\end{enumerate}
\end{defi}

We denote by $\widehat{\Div}(X)_\bR$ the $\bR$-vector space of adelic $\bR$-divisors. Since $X$ is normal, we have inclusions $\Div(X) \subseteq \Div(X)_\bQ \subseteq \Div(X)_\bR$. Therefore we can consider the subgroups $\widehat{\Div}(X),\widehat{\Div}(X)_\bQ$ of $\widehat{\Div}(X)_\bR$ defined by 
\[\widehat{\Div}(X) = \{\overline{D} = (D,(g_v)_{v \in \Sigma_K}) \in \widehat{\Div}(X)_\bR \ | \ D \in \Div(X)\}\]
and 
\[\widehat{\Div}(X)_\bQ = \{\overline{D} = (D,(g_v)_{v \in \Sigma_K}) \in \widehat{\Div}(X)_\bR \ | \ D \in \Div(X)_\bQ\}.\]
The elements of  $\widehat{\Div}(X)$ (respectively $\widehat{\Div}(X)_\bQ$) are called adelic Cartier divisors (respectively adelic $\bQ$-Cartier divisors) on $X$. 

\begin{example}\label{exampledeg} Let $(\xi_v)_{v\in \Sigma_K}$ be a collection of real numbers such that $\xi_v =0$ for all except finitely many $v \in \Sigma_K$. Then $\overline{\xi} = (0, (\xi_v)_{v\in \Sigma_K}) \in \widehat{\Div}(X)$. The  arithmetic degree of $\overline{\xi}$ is defined as 
\[\widehat{\deg}(\overline{\xi}):= \frac{1}{2}\sum_{v\in \Sigma_K}n_v(K)\xi_v. \]

\end{example}

\subsection{Small sections and arithmetic volume}\label{paragvolumes}  Let $\overline{D} = (D,(g_v)_{v\in\Sigma_K})$ be an adelic $\bR$-Cartier divisor on $X$, 
 and let $s = (\phi, D) \in \Rat(X,D)_\bR^\times$.  We let $\divi(s) = (\phi) + D \in \Div(X)_\bR$. For any  $z \in X_{\bC_v}^{\mathrm{an}}\setminus (\Supp ( \divi(s)))_{\bC_v}^{\mathrm{an}}$, we define
 \[\|s(z)\|_v = |\phi_v(z)|_v \exp(-g_v(z)/2),\]
 where $\phi_v$ is the pull-back of $\phi$ to $X_{\bC_v}$.
  If $s \in \Gamma(X, D)_\bR^\times$,  
 then $z \mapsto \|s(z)\|_v$ extends to a continuous function on  $X_{\bC_v}^{\mathrm{an}}$ (see \cite[Propositions 1.4.2 and 2.1.3]{MoriwakiMAMS}). In that case, we let $\|s\|_{v,\sup} = \sup_{z \in X_{\bC_v}^{\mathrm{an}}}\|s(z)\|_v$ and we say that $s$ is \emph{small} if $\|s\|_{v,\sup} \leq 1$ for all $v\in \Sigma_K$. We denote by $\widehat{\Gamma}(X, \overline{D})^\times$ (respectively $\widehat{\Gamma}(X, \overline{D})_\bR^\times$) the set of small non-zero global sections (respectively small non-zero $\bR$-global sections) of $D$.    We let $\widehat{\Gamma}(X, \overline{D}) = \widehat{\Gamma}(X, \overline{D})^\times \cup \{0\}$, 
 and we set 
\[\widehat{h}^0(X, \overline{D}) = \left\{
 \begin{tabular}{ll}
 $\ln \# \widehat{\Gamma}(X, \overline{D})$  & if $K_0 =  \bQ$,\\
 $\dim_k \widehat{\Gamma}(X, \overline{D})$  & if $K_0 = k(T)$.
 \end{tabular}
 \right.\]
  We define the \emph{arithmetic volume} of $\overline{D}$ as 
 \[\widehat{\vol}(\overline{D}) := \frac{1}{[K:K_0]}\limsup_{n \rightarrow \infty} \frac{\widehat{h}^0(X, n\overline{D})}{n^{d+1}/(d+1)!}.\]
The arithmetic volume satisfies the following continuity property, due to Moriwaki \cite[Theorem 5.2.1]{MoriwakiMAMS} (see also \cite[Theorem 6.4.24]{CMadelic}): for any adelic $\bR$-Cartier divisor $\overline{D}'\in \widehat{\Div}(X)_\bR$, we have
\begin{equation}\label{eqcontvol}
\lim_{\varepsilon \rightarrow 0} \widehat{\vol}(\overline{D} + \varepsilon\overline{D}') = \widehat{\vol}(\overline{D}).
\end{equation}
\subsection{Height function and semi-positivity}\label{sectionheightsemipos}
 Throughout this subsection, we fix an adelic $\bR$-Cartier divisor $\overline{D} = (D,(g_v)_{v\in\Sigma_K})$ on $X$.
\subsubsection{Height of points}\label{paragheightpoint}
  Let $x \in X_{\overline{K}}$ be a closed point, and let $K'$ be a finite extension of $K$ such that $x \in X(K')$. Let $w \in \Sigma_{K'}$ and let $\sigma_w \colon K'_w \hookrightarrow \bC_v$ be a $K_v$-embedding, where $v$ denotes the restriction of $w$ to $K$. The pair $(x, \sigma_w)$ uniquely determines  a point $x_w \in X_{\bC_v}(\bC_v)$. By \cite[§~4.2]{MoriwakiMAMS}, there exists $s \in \Rat(X,D)_\bR^\times$ such that $x \notin \Supp(\divi(s)_{K'})$. By Galois invariance of $g_v$, the quantity $\|s(x)\|_w:=\|s(x_w)\|_v$ does not depend on the choice of $\sigma_w$. We define the \emph{height of $x$} with respect to $\overline{D}$ as 
\[\widehat{h}_{\overline{D}}(x) = -\sum_{w \in \Sigma_{K'}} n_w(K') \ln \|s(x)\|_w.\]
It does not depend on the choices of $s$ and $K'$.
 
 \subsubsection{Semi-positivity and height of subvarieties}\label{paragheightvar}

In order to define the height of higher dimensional subvarieties, we recall the notion  of semi-positive adelic $\bR$-Cartier divisors used in \cite{BMPS}.

\begin{defi}\label{defisemipos} We say that $\overline{D}$ is semi-positive if for each $v \in \Sigma_K$, there exists a sequence $(g_{n,v})_{n\in \bN}$ satisfying the following conditions:
\begin{itemize}
\item if $v$ is Archimedean, $g_{n,v}$ is a smooth plurisubharmonic $D$-Green function invariant under complex conjugation for every $n \in \bN$;
\item if $v$ is non-Archimedean, then for every $n \in \bN$ there exists a relatively nef model $(\mathcal{X}_n, \mathcal{D}_n)$ of $(X_{K_v},D_{K_v})$ over the valuation ring of $K_v$ such that $g_{n,v} = g_{\mathcal{D}_{n},v}$;
\item for every $v \in \Sigma_K$, $(g_{n,v})_{n \in \bN}$ converges uniformly to $g_v$.
\end{itemize}

\end{defi}

\begin{rema} \begin{enumerate}[leftmargin=0.8cm]
\item Let $U \subseteq \mathcal{S}$ be a dense open subset as in Definition \ref{defiadelicCartier}, that is there exists a model $(\mathcal{X},\mathcal{D})$ of $(X,D)$ over $U$ such that $g_v = g_{\mathcal{D},v}$ for every $v \in U$.  If $\overline{D}$ is semi-positive, then $(\mathcal{X},\mathcal{D})$ is relatively nef by \cite[Corollary A.3.2]{MoriwakiMAMS}.
 In particular, $D$ is nef.

\item If $K$ is a number field, $\overline{D}$ is semi-positive if and only if it is relatively nef in the sense of \cite[Definition 4.4.1]{MoriwakiMAMS}. Moreover, our definition of semi-positivity coincides with the one of \cite[Definition 3.9]{BaEM} by \cite[Proposition 4.4.2]{MoriwakiMAMS}.

\end{enumerate}

\end{rema}

 Assume that $\overline{D}$ is semi-positive and let $Y \subseteq X$ be a subvariety. We denote by $h_{\overline{D}}(Y)$ the height of $Y$ with respect to $\overline{D}$ as defined in \cite[page 225]{BMPS}. By \cite[Proposition 1.5.10]{BPS14}, this definition is invariant by finite field extension; in particular, we can define the height  $h_{\overline{D}}(Y)$ of any subvariety $Y \subseteq X_{\overline{K}}$. If $Y \subseteq X_{\overline{K}}$ is a subvariety with $\deg_{D}(Y) \ne 0$, the \emph{normalized height} of $Y$ with respect to $\overline{D}$ is  defined as
\[ \widehat{h}_{\overline{D}}(Y) = \frac{{h}_{\overline{D}}(Y)}{(\dim Y + 1)\deg_{D}(Y)}.\]
We gather some basic properties of heights in the following remark.

\begin{rema}\label{remaheight}  \begin{enumerate}[leftmargin=0.8cm]
\item If $x \in X_{\overline{K}}$ is a closed point, then $ \widehat{h}_{\overline{D}}(\{x\})= h_{\overline{D}}(\{x\})= \widehat{h}_{\overline{D}}(x)$ coincides with the height of $x$ defined in §~\ref{paragheightpoint} (here $\{x\}$ is considered as a subvariety of $X_{\overline{K}}$).
\item\label{rematwistheight} Let $\overline{\xi} = (0, (\xi_v)_{v \in \Sigma_K}) \in \widehat{\Div}(X)$ be such that $\widehat{\deg}(\overline{\xi} )=1$, 
 and let $\overline{D}(t) := \overline{D} - t\overline{\xi}$ for some real number $t \in \bR$. For any subvariety $Y\subseteq X_{\overline{K}}$, we have 
\[h_{\overline{D}(t)}(Y) = h_{\overline{D}}(Y) - t(\dim Y + 1)\deg_D(Y).\]
In particular, if $\deg_D(Y) \ne 0$ then $\widehat{h}_{\overline{D}(t)}(Y) = \widehat{h}_{\overline{D}}(Y) -t$. This follows from the Bézout formula \cite[(3.13) page 225]{BMPS} by induction on $\dim Y$.
\item\label{remaheightgeom} Assume that $K$ is a function field and that there exists a $C_K$-model $(\mathcal{X}, \mathcal{D})$ of $(X,D)$ with $g_v = g_{\mathcal{D},v}$ for every place $v \in \Sigma_K$. Then for any subvariety $\mathcal{Y} \subseteq \mathcal{X}$ such that $Y:= \mathcal{Y} \cap X \ne \emptyset$, we have $h_{\overline{D}}(Y) = \mathcal{D}^{\dim \mathcal{Y}} \cdot \mathcal{Y}$. 
\end{enumerate}
\end{rema}

\subsection{Positivity of adelic \texorpdfstring{$\bR$}{R}-Cartier divisors}\label{paragpositivitydef}
Let $\overline{D} = (D, (g_v)_{v \in \Sigma_K})$ be an adelic $\bR$-Cartier divisor on $X$.  We define the \emph{absolute minimum of $\overline{D}$} as 
\[\zeta_{\mathrm{abs}}(\overline{D}) = \inf_{x \in X_{\overline{K}}} \widehat{h}_{\overline{D}}(x) \in \bR \cup \{- \infty\}.\]
We recall below different notions of arithmetic positivity, following \cite[Definition 3.18]{BMPS}.

\begin{defi}\label{defipositivity}
  We say that $\overline{D}$ is
\begin{enumerate}
\item \emph{big} if $\widehat{\vol}(\overline{D}) > 0$;
\item  \emph{pseudo-effective} if $\overline{D} + \overline{A}$ is big for any big adelic $\bR$-Cartier divisor $\overline{A}$ on $X$;
\item \emph{nef} if $\overline{D}$ is semi-positive and if $\zeta_{\mathrm{abs}}(\overline{D}) \geq 0$;
\item \emph{generated by strictly small $\bR$-sections} if for every $x\in X_{\overline{K}}$, there exists $s \in \widehat{\Gamma}(X,\overline{D})_\bR^\times$ such that $x \notin \Supp(\divi(s)_{\overline{K}})$ and 
\[\sum_{v \in \Sigma_K} n_v(K) \ln \|s\|_{v,\sup} < 0;\]
\item \emph{ample} if the following three conditions hold: 
\begin{itemize}
\item $D$ is ample,
\item $\overline{D}$ is semi-positive, 
\item $\overline{D}$ is generated by strictly small $\bR$-sections.
\end{itemize}
\end{enumerate}
\end{defi}

\begin{rema} By \cite[Proposition 3.23]{BMPS}, our definition of pseudo-effective adelic $\bR$-Cartier divisors coincides with the one of \cite[Definition 3.18]{BMPS}.

\end{rema}

The following generalization of the arithmetic Hodge index Theorem is due to Moriwaki.

\begin{theorem}[Moriwaki]\label{thmHodge} If $\overline{D}$ is semi-positive, then $h_{\overline{D}}(X) \leq \widehat{\vol}(\overline{D})$, with equality if $\overline{D}$ is nef.  
\end{theorem}

\begin{proof} When $K$ is a number field, this is \cite[Theorem 5.3.2]{MoriwakiMAMS}. When $K$ is a function field, by arguing as in \cite{MoriwakiMAMS} it is enough to prove the result under the additional following assumption: $D \in \Div(X)$ and 
 there exists a normal $C_K$-model $(\mathcal{X},\mathcal{D})$ of $(X,D)$ such that $\mathcal{D} \in \Div(\mathcal{X})$ is a  relatively ample Cartier divisor on $\mathcal{X}$ with $g_v = g_{\mathcal{D},v}$ for every $v \in \Sigma_K$. In that case the result follows from the asymptotic Riemann-Roch theorem as in \cite[Proof of Theorem 7.3]{BaEM}.
\end{proof}

\begin{coro}\label{lemmaminabsinfsubvar} For any semi-positive $\overline{D} \in \widehat{\Div}(X)_\bR$, we have 
 \[\zeta_{\mathrm{abs}}(\overline{D}) = \inf_{\substack{Y \subseteq X_{\overline{K}}\\ \deg_D(Y)>0}} \widehat{h}_{\overline{D}}(Y),\]
where the infimum is over the subvarieties $Y \subseteq X_{\overline{K}}$ such that $\deg_D(Y)>0$.
\end{coro}

\begin{proof}
 Assume that $\zeta:=\zeta_{\mathrm{abs}}(\overline{D})\in \bR$, and 
 define $\overline{D}(\zeta)$ as in Remark \ref{remaheight} \eqref{rematwistheight}. Then $\widehat{h}_{\overline{D}(\zeta)}(Y) = \widehat{h}_{\overline{D}}(Y) - \zeta$ for any subvariety $Y \subseteq X_{\overline{K}}$ with $\deg_D(Y)>0$. In particular, $\overline{D}(\zeta)$ is nef and thus $\widehat{h}_{\overline{D}(\zeta)}(Y) \geq 0$ by Theorem \ref{thmHodge}. Therefore
\[\inf_{\substack{Y \subseteq X_{\overline{K}}\\ \deg_D(Y)>0}} \widehat{h}_{\overline{D}}(Y) \geq \zeta_{\mathrm{abs}}(\overline{D}).\]
The other inequality is obvious, since $\deg_D(\{x\}) = 1$ for every $x \in X_{\overline{K}}$.
\end{proof}

 The following theorem is an arithmetic version of the Nakai--Moishezon criterion. Over a number field, it was originally proved by Zhang \cite[Corollary 4.8]{Zhangplav} for hermitian line bundles, and it was recently generalized to adelic $\bR$-Cartier divisors in \cite{BaNM}.
\begin{theorem}\label{thmNM} The following conditions are equivalent:
\begin{enumerate}
\item\label{itemcritample} $\overline{D}$ is ample;
\item\label{itemcritampleminpts} $\overline{D}$ is semi-positive, $D$ is ample and $\zeta_{\mathrm{abs}}(\overline{D}) > 0$;
\item\label{itemcritampleminsubv} $\overline{D}$ is semi-positive, $D$ is ample and $\widehat{h}_{\overline{D}}(Y) > 0$ for every subvariety $Y \subseteq X_{\overline{K}}$.
\end{enumerate}

\end{theorem}

\begin{proof} The implication $\eqref{itemcritample} \Rightarrow \eqref{itemcritampleminsubv}$  
 follows from the Bézout formula \cite[(3.13) page 225]{BMPS} by induction on $\dim Y$. We shall prove that $\eqref{itemcritampleminsubv}\Rightarrow \eqref{itemcritampleminpts}$ by induction on $d$. Since $\widehat{h}_{\overline{D}}(X) > 0$ by assumption, there exists a closed subscheme $Y \varsubsetneq X_{\overline{K}}$ such that $\inf_{x \in X_{\overline{K}} \setminus Y} \widehat{h}_{\overline{D}}(x) > 0$ by \cite[Theorem 1.5]{BaEM}. On the other hand $\inf_{x \in Y} h_{\overline{D}}(x) > 0$ by the induction hypothesis, and therefore \eqref{itemcritampleminpts} holds. 
   If $K$ is a number field, the implication $\eqref{itemcritampleminpts} \Rightarrow \eqref{itemcritample}$ is a direct consequence of \cite[Theorem 1.1]{BaNM}. Assume that $K$ is a function field.    We shall prove that $\eqref{itemcritampleminpts} \Rightarrow \eqref{itemcritample}$ under the following additional assumption:  
  there exists a relatively nef model $(\mathcal{X}, \mathcal{D})$ of $(X,D)$ over $C_K$ such that $g_v = g_{\mathcal{D},v}$ for every $v \in \Sigma_K$. The general case follows by semi-positivity. Let $\zeta$ be a positive real number with $\zeta < \zeta_{\mathrm{abs}}(\overline{D})$. Then $\widehat{h}_{\overline{D}}(Y) > \zeta$ for any subvariety $Y \subseteq X_{\overline{K}}$ by Corollary \ref{lemmaminabsinfsubvar}. Let $\mathcal{F} \in \Div(\mathcal{X})$ be a general fiber of $\mathcal{X} \rightarrow C_K$ and let  
    $\mathcal{Y} \subseteq \mathcal{X}$ be subvariety such that 
   $Y := \mathcal{Y} \cap X_{\overline{K}}\ne \emptyset$. By Remark \ref{remaheight} \eqref{rematwistheight} and \eqref{remaheightgeom}, we have
  \begin{equation*}
  ( \mathcal{D} - \zeta \mathcal{F})^{\dim \mathcal{Y}} \cdot \mathcal{Y} = (\dim \mathcal{Y})\deg_D(Y)(\widehat{h}_{\overline{D}}(Y) - \zeta) > 0.
\end{equation*}  
By \cite[Theorem 1.3]{Birkar}, the set
\[\bigcap_{s \in \Gamma(\mathcal{X}, \mathcal{D} - \zeta \mathcal{F})_\bR^\times} \Supp(\divi(s))\]
 does not intersect $X$. Therefore, for every point $x \in X_{\overline{K}}$, there exists $s \in \Gamma(\mathcal{X},\mathcal{D} - \zeta \mathcal{F})_\bR^\times$ such that $x \notin \Supp(\divi(s_{|X}))$. By construction, such $s_{|X} \in \Gamma(X,D)^\times_\bR$ is a stricly small $\bR$-section for $\overline{D}$ and we are done.
\end{proof}

\begin{coro}\label{coroheightmin} Assume that $D$ is ample and that $\overline{D}$ is semi-positive. Then there exists a subvariety $Y \subseteq X_{\overline{K}}$ such that 
\[\zeta_{\mathrm{abs}}(\overline{D}) = \widehat{h}_{\overline{D}}(Y). \]
\end{coro}

\begin{proof} Let $\zeta = \zeta_{\mathrm{abs}}(\overline{D})$,
and define $\overline{D}(\zeta)$ as in Remark \ref{remaheight} \eqref{rematwistheight}. Then $\zeta_{\mathrm{abs}}(\overline{D}(\zeta)) = \zeta_{\mathrm{abs}}(\overline{D}) - \zeta = 0$. By Theorem \ref{thmNM}, there exists a subvariety $Y \subseteq X_{\overline{K}}$ such that  $\widehat{h}_{\overline{D}}(Y) - \zeta_{\mathrm{abs}}(\overline{D}) = \widehat{h}_{\overline{D}(\zeta)}(Y)= 0$.
\end{proof}

\section{Cauchy's inequality for adelic Cartier divisors}\label{sectionCauchy}

The goal of this section is to prove Proposition \ref{propCauchy} below, which is the main new ingredient of this paper.  It can be considered as an adelic version of Cauchy's inequality, as it gives a lower bound for the height of a closed point $x \in X_{\overline{K}}$ in terms of the height of a global section and its order of vanishing at $x$.

\subsection{Differentiation of global sections}\label{paragdiffsections} Let  $\mathbf{z} = (z_1, \ldots, z_d)$ be a system of parameters centered at a closed regular point  $x \in X_{\overline{K}}$, that is a system of  generators  with $d$ elements of the maximal ideal of $\cO_{X_{\overline{K},x}}$.  For any $\bm{\alpha} = (\alpha_1, \ldots, \alpha_d) \in \bN^d$, we let $\mathbf{z}^{\bm{\alpha}} = z_1^{\alpha_1} \cdots z_d^{\alpha_d}$ and 
 $|\bm{\alpha}| := \alpha_1 + \cdots + \alpha_d$.   Since $x$ is regular, the completion $\widehat{\cO}_{X_{\overline{K}},x} $ of the local ring ${\cO}_{X_{\overline{K}},x}$ is naturally isomorphic to $\overline{K} [[z_1, \ldots z_d]]$.
  Therefore,  any $f \in \cO_{X_{\overline{K}},x}$ has a unique expression
 \[f = \sum_{\bm{\alpha} \in \bN^d} \bm{\partial}_{\mathbf{z}}^{\bm{\alpha}}f(x) \mathbf{z}^{\bm{\alpha}},\]
 where $\bm{\partial}_{\mathbf{z}}^{\bm{\alpha}}f(x) \in \overline{K}$ for every $\bm{\alpha} \in \bN^d$. Let $D \in \Div(X)$ and let $s_0 \in \Gamma (U,D_{\overline{K}})$ be a trivialization of $\cO_X(D)_{\overline{K}}$ on a neighbourhood $U \subseteq X_{\overline{K}}$ of $x$.
 Given a non-zero global section $s \in \Gamma(X,D)^\times$, there exists $f \in \cO_{X_{\overline{K}}}(U)$ such that $s_{|U} = f \cdot s_0$. We define the order of $s$ at $x$ by
\[\ord_x s = \min \{|\bm{\alpha}| \ | \ \bm{\partial}_{\mathbf{z}}^{\bm{\alpha}}f(x)\ne 0\}.\] 
Note that $\ord_xs$ does not depend on the choice of the system of parameters $\mathbf{z}$. By Leibniz's formula, it is also independent of $s_0$.

\subsection{Height of global sections} Let $\overline{D} = (D,(g_v)_{v\in \Sigma_K}) \in \widehat{\Div}(X)$ and let $s =(\phi, D_{K'}) \in \Gamma(X_{K'}, D_{K'})^\times$, where $K'$ is a finite extension of $K$. Let $w \in \Sigma_{K'}$ and $\sigma_w \colon K'_w \hookrightarrow \bC_v$ be a $K_v$-embedding, where $v$ is the restriction of $w$ to $K$. We denote by $\widetilde{\sigma}_w \colon X_{\bC_v} \rightarrow X_{K'}$ the morphism induced by $\sigma_w$. 
 The function 
\[z \in X_{\bC_v}^{\mathrm{an}} \setminus \Supp(\widetilde{\sigma}_w^*\divi(s))^{\mathrm{an}} \mapsto \|s(z)\|_w :=  |\widetilde{\sigma}_w^*\phi(z)|_v\exp(-g_v(z)/2) \in \bR\] 
extends to $X_{\bC_v}^{\mathrm{an}}$ and does not depend on the choice of $\sigma_w$. We let $\|s\|_{w,\sup} = \sup_{z\in X_{\bC_v}^{\mathrm{an}}} \|s(z)\|_w$ and  we define the \emph{height of $s$ with respect to $\overline{D}$} as 
\[h_{\overline{D}}(s) = \sum_{w \in \Sigma_{K'}}n_w(K') \ln\|s\|_{w,\sup}.\] 
This definition is invariant by finite field extension. In particular,  we can define the height $h_{\overline{D}}(s)$ of any global section $s \in \Gamma(X_{\overline{K}},D_{\overline{K}})^\times$.

\subsection{Cauchy's inequality for global sections}\label{paragCauchyineq}  

Given a closed point $x \in X_{\overline{K}}$ and an adelic Cartier divisor $\overline{D} \in \widehat{\Div}(X)$, it follows from the definitions that $\widehat{h}_{\overline{D}}(x) \geq -h_{m\overline{D}}(s)/m$ for any integer $m\geq 1$ and any global section $s \in \Gamma(X_{\overline{K}},mD_{\overline{K}})^\times$ such that $s(x) \ne 0$. It is natural to ask whether a similar result holds even if $s$ vanishes at $x$. The following proposition gives an answer to this question, that will be central in our proof of Theorem \ref{thmnefintro}.
\begin{prop}\label{propCauchy}  Let $\overline{D} = (D, (g_v)_{v \in \Sigma_K}) \in \widehat{\Div}(X)$ be an adelic Cartier divisor, and assume that $D$ is ample. Let $x \in X_{\overline{K}}$ be a regular closed point and let $\varepsilon > 0$ be a real number. There exists a real number $\rho(\overline{D},x,\varepsilon)$ such that
\begin{equation}\label{eqkeyineq}
 \widehat{h}_{\overline{D}}(x) + \varepsilon \geq - \frac{1}{m} ( h_{m\overline{D}}(s) + \rho(\overline{D},x,\varepsilon)\ord_x s)
\end{equation}
 for any positive integer $m$ and for any $s \in \Gamma(X_{\overline{K}}, m D_{\overline{K}})^\times$.
\end{prop}

\begin{proof} We assume that $x \in X(K)$ is a $K$-rational point without loss of generality.
After possibly replacing $D$ by a suitable multiple, we also assume that $\cO_X(D)$ is very ample.  Let $\varphi \colon X \hookrightarrow \bP^N_K$ be a closed immersion such that $\cO_X(D) = \varphi^*\cO_{\bP^N_K}(1)$. By Bertini's theorem \cite[Example 5.2.14]{LazI}, there exist homogeneous coordinates $(T_0, \ldots, T_N)$ on $\bP^N_K$ such that $x \notin U := X\setminus \divi(T_0)$ and  the projection on the first $d+1$ coordinates induces a well-defined finite morphism $q_X  \colon X \rightarrow \bP^d_K$, étale at $x$.  We denote by $q \colon U(\overline{K}) \rightarrow \overline{K}^d$ the restriction of $q_X$ to $U(\overline{K})$, and we write $q(x) = (x_1, \ldots, x_d) \in K^d$. For any $v \in \Sigma_K$, $q$ induces a morphism $q_v \colon U_{\bC_v}(\bC_v) \rightarrow \bC_v^d$, and given a real number $\rho > 0$ we let 
\[\mathbb{D}_v(\rho) = \{z = (z_1, \ldots, z_d) \in \bC_v^d \ | \ \max_{1 \leq i \leq d} |z_i - x_i|_v < \rho\}.\]
 Moreover,  we let $s_i = \varphi^*T_i \in \Gamma(X,D)$ for each $i \in \{0, \ldots, d\}$.  Let us prove the following claim.
\begin{claim}\label{claimtriv} There exists a finite subset $\mathcal{P}_1 \subset \Sigma_K$ such that $\|s_0(y)\|_v = 1$ for every $v \in \Sigma_K \setminus \mathcal{P}_1$ and every $y \in q_v^{-1}(\mathbb{D}_v(1))$.
\end{claim}
\begin{proof}
 Let $(\mathcal{X}, \mathcal{D})$ be a model of $(X,D)$ over a dense open subset $\mathcal{U} \subseteq \mathcal{S}$ such that $g_v = g_{\mathcal{D},v}$ for every $v \in \mathcal{U}$. After possibly shrinking $\mathcal{U}$, there exist global sections $\mathfrak{s}_0, \ldots, \mathfrak{s}_d \in \Gamma(\mathcal{X}, \mathcal{D})$ such that ${\mathfrak{s}_i}_{|X} = s_i$ for every $i \in \{0, \ldots, d\}$ and $\cap_{i=0}^d \divi(\mathfrak{s}_i) = \emptyset$.   The sections $\mathfrak{s}_0, \ldots, \mathfrak{s}_d$ induce a morphism $q_{\mathcal{X}} \colon \mathcal{X} \rightarrow \mathbb{P}^d_U$ that extends $q_X \colon X \rightarrow \bP^d_K$. Let $v \in \mathcal{U}$ and let $q_{X,v} \colon X_{\bC_v} \rightarrow \bP^d_{\bC_v}$ be the base change of $q_X\colon X \rightarrow \bP^d_K$. We denote by $\bC_v^{\circ}$ the valuation ring of $\bC_v$ and by $\bC_v^{\circ\circ}$ its maximal ideal. For any $i \in \{0, \ldots, d\}$,    we let $\mathfrak{s}_{i,v} \in \Gamma(\mathcal{X}_{\bC_v^\circ / \bC_v^{\circ \circ}}, \mathcal{D}_{\bC_v^\circ / \bC_v^{\circ \circ}})$ be the pullback of $\mathfrak{s}_i$ to $\mathcal{X}_{\bC_v^\circ / \bC_v^{\circ \circ}} = \mathcal{X} \times_{\mathcal{U}} \Spec(\bC_v^\circ / \bC_v^{\circ \circ})$. We have a commutative diagram
  \[ \begin{tikzcd}[column sep=large]
X_{\bC_v}  \arrow{d}{q_{X,v}} \arrow{r}{\mathrm{red}_{X,v}} & \mathcal{X}_{\bC_v^\circ / \bC_v^{\circ \circ}} \arrow{d}{q_{\mathcal{X},v}}   \\
\bP^d_{\bC_v}  \arrow{r}{\mathrm{red}_{\bP^d,v}}& \bP^d_{\bC_v^\circ / \bC_v^{\circ \circ}},
\end{tikzcd}
\] 
 where $ \mathrm{red}_{X,v}$, $\mathrm{red}_{\bP^d,v}$ are the reduction maps (see \cite[section 1.3]{MoriwakiMAMS}) and $q_{\mathcal{X},v}$ is the morphism given by $z \mapsto (\mathfrak{s}_{0,v}(z) : \cdots :  \mathfrak{s}_{d,v}(z))$.
   Let $y \in q_v^{-1}(\mathbb{D}_v(1))$. Then by definition of $\mathrm{red}_{\bP^d,v}$ we have $\mathrm{red}_{\bP^d,v}(q_{X,v}(y)) = (1 : y_{1,v} : \cdots : y_{d,v})$ for some $(y_{1,v}, \ldots, y_{d,v}) \in (\bC^\circ / \bC_v^{\circ \circ})^d$. On the other hand, 
\[\mathrm{red}_{\bP^d,v}(q_{X,v}(y)) = q_{\mathcal{X},v}(\mathrm{red}_{X,v}(y)) = (\mathfrak{s}_{0,v}(\mathrm{red}_{X,v}(y)) : \cdots :  \mathfrak{s}_{d,v}(\mathrm{red}_{X,v}(y)))\]   
and therefore  $\mathfrak{s}_{0,v}(\mathrm{red}_{X,v}(y)) \ne 0$. Since $g_v = g_{\mathcal{D},v}$, it follows that $\|s_0(y)\|_v = 1$ by definition of $g_{\mathcal{D},v}$ (see \cite[section 0.2]{MoriwakiMAMS}). The claim follows by taking $\mathcal{P}_1 = \Sigma_K \setminus (\mathcal{U} \cup \Sigma_{K,\infty})$. 
  \end{proof} 
  
  We also need the following claim, which is a direct consequence of a multivariable version of Hensel's lemma due to Vojta \cite[Corollary 15.13]{Vojta}.

\begin{claim}\label{claimsection} There exists a collection of positive real numbers $(\rho_v)_{v\in \Sigma_K}$ and a finite set $\mathcal{P}_2 \subset \Sigma_K$ such that 
\begin{enumerate}
\item\label{itemsection} for every $v \in \Sigma_K$, the map $q_v$ admits an analytic section $\sigma_v \colon \mathbb{D}_v(\rho_v) \rightarrow U_{\bC_v}(\bC_v)$ on $\mathbb{D}_v(\rho_v)$, and
\item $\rho_v =1$ for every $v \in \Sigma_K \setminus \mathcal{P}_2$.
\end{enumerate}
\end{claim}
\begin{proof}  Since $q$ is étale at $x$, there exist an integer $n \in \bN$ and polynomials $f_1, \ldots, f_n \in K[T_1, \ldots, T_d, S_1, \ldots, S_n]$ such that 
\[U \cong \Spec \left( \frac{K[T_1, \ldots, T_d,S_1, \ldots, S_n]}{(f_1, \ldots, f_n)} \right),\]
and \[J_f(x) := ((\partial f_j/\partial S_j)(x))_{1 \leq i,j \leq n} \in \mathrm{GL}_n(K).\] 
  By the multivariable Hensel's lemma \cite[Corollary 15.13]{Vojta}, for every place $v \in \Sigma_K$ there exists a real number $\rho_v > 0$ together with an analytic section 
 \[\sigma_v \colon \mathbb{D}_v(\rho_v) \longrightarrow U_{\bC_v}(\bC_v)\]
  of $q_v$, and moreover we can take  $\rho_v = |\det J_f(x) |_v^2$ for all except finitely many places $v$ (note that the proof of \cite[Corollary 15.13]{Vojta} remains valid in the function field case, since \cite[Corollary 15.3]{Vojta} deals with an arbitrary non-Archimedean valued field). 
   Since $\det J_f(x) \in K^{\times}$, there exists a finite set $\mathcal{P}_2 \subset \Sigma_K$ such that $\rho_v =1$ for every $v \in \Sigma_K \setminus \mathcal{P}_2$. 
 
\end{proof}

Let $\mathcal{P}_1, \mathcal{P}_2 \subset \Sigma_K$ and $(\rho_v)_{v\in \Sigma_K}$  be as in Claims  \ref{claimtriv}, \ref{claimsection} and let $\mathcal{P} = \mathcal{P}_1 \cup \mathcal{P}_2$. For every $v \in \Sigma_K$ and $\rho \in (0, \rho_v]$, we let $\mathbb{B}_v(\rho) = \sigma_v(\mathbb{D}_v(\rho))$ and 
\[\alpha_v(\rho) = \sup_{y \in \mathbb{B}_v(\rho)} \frac{\|s_0(x)\|_v}{\|s_0(y)\|_v}.\]
By Claims \ref{claimtriv} and \ref{claimsection} \eqref{itemsection}, we have $\rho_v = 1$ and $\alpha_v(1) = 1$ for every $v \notin \mathcal{P}$. Let $\varepsilon > 0$.  By continuity of the metrics, for every $v \in \mathcal{P}$ there exists a positive real number $\rho_{\varepsilon,v} \leq \rho_v$ such that $\alpha_v(\rho_{\varepsilon,v}) \leq \exp(\varepsilon/(n_v(K)\#\mathcal{P}))$. For every $v \notin \mathcal{P}$, we  put $\rho_{\varepsilon,v} = 1$. We have 
\[\sum_{v \in \Sigma_K} n_v(K) \ln \alpha_v(\rho_{\varepsilon,v}) = \sum_{v \in \mathcal{P}} n_v(K) \ln \alpha_v(\rho_{\varepsilon,v})  \leq \varepsilon, \]
and we define \[\rho(\overline{D},x,\varepsilon) = - \sum_{v \in \Sigma_K} n_v(K) \ln \rho_{\varepsilon,v}= -\sum_{v \in \mathcal{P}} n_v(K) \ln \rho_{\varepsilon,v} \in \bR.\]

Let $m \geq 1$ be an integer and let $s \in \Gamma(X_{\overline{K}}, m D_{\overline{K}})^\times$. In order to prove \eqref{eqkeyineq}, we assume that $s \in \Gamma(X ,mD)^\times$ is defined over $K$ without loss of generality. For every $i \in \{1, \ldots, d\}$, we let $z_i = s_i/s_0 - x_i \in \cO_X(U)$ (recall that $s_i = \varphi^*T_i$). By construction,  $ \mathbf{z} = (z_1, \ldots, z_d)$ is a system of parameters at $x$; by a slight abuse of notation, we also consider it as a system of parameters at $q(x)$. Let $f = s/s_0^m \in \cO_X(U)$.  By definition of $\ord_x s$, there exists $\bm{\alpha} \in \bN^d$ with $|\bm{\alpha}| = \ord_x s$ such that $ \bm{\partial}_{\mathbf{z}}^{\bm{\alpha}} f (x) \ne 0$.
Let $v \in \Sigma_K$ and let $f_v \in \cO_{X_{\bC_v}}(U_{\bC_v})$ be the pull-back of $f$ to $U_{\bC_v}$.  By Cauchy's inequality in several variables (see \cite[§~5.1.3, Proposition 3]{BGR} in the non-Archimedean case), we have 
\begin{align*}
|\bm{\partial}_{\mathbf{z}}^{\bm{\alpha}} f(x)|_v = |\bm{\partial}_{\mathbf{z}}^{\bm{\alpha}} (\sigma_v^*f_{v})(q_v(x))|_v & \leq \frac{\sup_{z \in \mathbb{D}_v(\rho_{\varepsilon,v})} | \sigma_v^*f_v(z)|_v}{\rho_{\varepsilon,v}^{\ord_x s}}\\
&   = \frac{\sup_{z \in \mathbb{B}_v(\rho_{\varepsilon,v})} | f_{v}(z)|_v}{\rho_{\varepsilon,v}^{\ord_x s}},
\end{align*}  
where $\sigma_v \colon \mathbb{D}_v(\rho_v) \rightarrow U_{\bC_v}(\bC_v)$ is the section given by Claim \ref{claimsection} \eqref{itemsection}. Therefore 
\[\|s_0(x) \|_v^m |\bm{\partial}_{\mathbf{z}}^{\bm{\alpha}} f(x)|_v \leq \alpha_v(\rho_{\varepsilon,v})^m\frac{\sup_{z \in \mathbb{B}_v(\rho_{\varepsilon,v})} \|s(z)\|_v}{\rho^{\ord_x s}} \leq  \alpha_v(\rho_{\varepsilon,v})^m\frac{\|s\|_{v,\sup}}{\rho_{\varepsilon,v}^{\ord_x s}}\]
for any $v \in \Sigma_K$. Since $ \bm{\partial}_{\mathbf{z}}^{\bm{\alpha}} f (x) \ne 0$, we have
 \begin{align*}
 m\widehat{h}_{\overline{D}}(x)& =  - \sum_{v \in \Sigma_K} n_v(K)\ln( \|s_0(x) \|_v^m |\bm{\partial}_{\mathbf{z}}^{\bm{\alpha}} f(x)|_v)  \\
  & \geq  -  h_{m\overline{D}}(s) - \rho(\overline{D},x,\varepsilon)\ord_x s - m\sum_{v \in \Sigma_K} n_v(K) \ln \alpha_v(\rho_{\varepsilon,v}) \\
  & \geq  -  h_{m\overline{D}}(s) - \rho(\overline{D},x,\varepsilon)\ord_x s - m\varepsilon
 \end{align*}
 and we are done.

\end{proof}

 \section{Concave transform and arithmetic Okounkov bodies}\label{sectionBCarithmNO}

 We briefly recall the construction of geometric Okounkov bodies in subsection \ref{paragNO}, following \cite{LazMus}, \cite[§~7.3]{MoriwakiMAMS} and \cite{BoucksomOkounkov}. We then define  the arithmetic concave transform (§~\ref{paragconcave}) introduced by Boucksom and Chen \cite{BC}. The only new material in this section is contained in §§~\ref{paraggeneralconc} and \ref{paragarithmNO}, where we generalize the construction of the concave transform and the arithmetic Okounkov bodies from \cite{BC} to arbitrary adelic $\bR$-Cartier divisors (without assuming the underlying Cartier divisor to be big). 
 
\subsection{Okounkov bodies of big Cartier divisors}\label{paragNO}

We denote by $\leq_{\mathrm{lex}}$ the lexicographic order on $\bZ^d$.
\begin{defi}\label{defivaluation} A valuation of maximal rank $\nu$ on $\Rat(X_{\overline{K}})$ is a surjective map $\nu \colon \Rat(X_{\overline{K}})^\times \rightarrow \bZ^d$ satisfying the following three conditions:
\begin{itemize}
\item $\nu(fg) = \nu(f)+\nu(g)$ for every $f,g \in \Rat(X_{\overline{K}})^\times$,
\item $\min \{\nu(f), \nu(g) \} \leq_{\mathrm{lex}} \nu(f+g)$ for every $f,g \in \Rat(X_{\overline{K}})^\times$ with $f+g \ne 0$,
\item $\nu(a) = 0_{\bZ^d}$ for every $a \in \overline{K}^\times$.
\end{itemize}
We denote by $\mathcal{V}(X_{\overline{K}})$ the set of valuations of maximal rank on $\Rat(X_{\overline{K}})^\times$. For any $\nu \in \mathcal{V}(X_{\overline{K}})$, we let $\nu(0) = \infty$, with the convention that $\bm{\alpha} <_{\mathrm{lex}} \infty$ for every $\bm{\alpha} \in \bZ^d$.
\end{defi}

Let $\nu \in \mathcal{V}(X_{\overline{K}})$. We consider the sets 
\[\cO_{\nu} = \{ f \in \Rat(X_{\overline{K}}) \ | \ \nu(f) \geq_{\mathrm{lex}} 0_{\bZ^d}\} \ \text{ and } \ \mathfrak{m}_{\nu} = \{ f \in \cO_\nu \ |  \nu(f) >_{\mathrm{lex}}0_{\bZ^d}\}. \]
 By the valuative criterion for properness, there exists a unique schematic point $c_X({\nu}) \in X_{\overline{K}}$ such that 
 \[\cO_{X_{\overline{K}},c_{X}(\nu)} \subseteq \cO_{\nu} \ \text{ and } \ \mathfrak{m}_{X_{\overline{K}},c_{X}(\nu)}  = \cO_{X_{\overline{K}},c_X(\nu)} \cap \mathfrak{m}_{\nu},\] where $\mathfrak{m}_{X_{\overline{K}},c_{X}(\nu)}$ denotes the maximal ideal of $\cO_{X_{\overline{K}},c_{X}(\nu)}$.
 We call $c_{X}(\nu)$ the center of $\nu$.  By \cite[Remark 2.25]{BoucksomOkounkov}, we have  $\cO_{\nu}/\mathfrak{m}_{\nu} = \overline{K}$. In particular, $c_{X}(\nu) \in X_{\overline{K}}$ is a closed point. Moreover, for every closed point $x \in X_{\overline{K}}$ there exists a valuation of maximal rank $\nu \in \mathcal{V}(X_{\overline{K}})$ with center $c_{X}(\nu) = x$ (see \cite[Chap. VI, § 16, Theorem 37, page 106]{ZS}).

\begin{example}\label{examplevaladm}  Let $Y_\bullet$ be an admissible flag on $X_{\overline{K}}$, that is a sequence 
\[Y_\bullet\colon X_{\overline{K}} = Y_0 \supsetneq Y_1 \supsetneq \cdots \supsetneq Y_{d-1} \supsetneq Y_{d} = \{x\}\]
such that for every $i \in \{0, \ldots, d\}$, $Y_i \subseteq X_{\overline{K}}$ is a subvariety of codimension $i$, smooth at $x$. We define a valuation of maximal rank $\nu_{Y_\bullet} \in \mathcal{V}(X_{\overline{K}})$ as follows. For each $i \in \{1, \ldots, d-1\}$, we choose a local equation $\omega_i \in \cO_{Y_i, x}$ of $Y_{i+1}$, and for every $f \in \Rat(X_{\overline{K}})^\times$ we let 
\[\nu_{Y_\bullet}(f)= (\ord_{Y_1}(f_0), \ord_{Y_2}(f_1), \ldots, \ord_{Y_d}(f_d)),\]
where the $f_i  \in \Rat(Y_i)^\times$ are defined inductively by $f_0 = f$ and 
\[f_{i+1} = (\omega_{i+1}^{-\ord_{Y_{i+1}}(f_i)} f_i)_{|Y_{i+1}}.\]
Note that $c_X(\nu_{Y_\bullet}) = x \in X_{\overline{K}}$. 
\end{example} 

\begin{example}\label{examplevalinf} Let $x \in X_{\overline{K}}$ be a smooth closed point and let $\widetilde{X} \rightarrow X$ be the blow-up of $X_{\overline{K}}$ at $x$, with exceptional divisor $E$. An infinitesimal flag over $x\in X_{\overline{K}}$ is an admissible flag on $\widetilde{X}$
\[\widetilde{Y}_\bullet\colon \ \widetilde{X} = \widetilde{Y}_0 \supsetneq \widetilde{Y}_1 \supsetneq \cdots \supsetneq \widetilde{Y}_{d-1} \supsetneq \widetilde{Y}_{d}\]
such that  $E =  \widetilde{Y}_1$ and for each $i \in \{2, \ldots, d\}$, $\widetilde{Y}_i$ is an $(d-i)$-dimensional linear subspace of $E   \simeq \bP^{d-1}_{\overline{K}}$ for every $i \in \{1, \ldots, d\}$. We denote by $\nu_{\widetilde{Y}_\bullet} \in \mathcal{V}(X_{\overline{K}})$  the valuation of maximal rank on $\Rat(X_{\overline{K}}) = \Rat(\widetilde{X})$ constructed in Example \ref{examplevaladm}. We have   $c_X(\nu_{\widetilde{Y}_\bullet}) = x$.
\end{example}
In the rest of this section, we fix a valuation of maximal rank $\nu \in \mathcal{V}(X_{\overline{K}})$. 
 For a Cartier divisor $D \in \Div(X_{\overline{K}})$ on $X_{\overline{K}}$, 
 we define $\nu(D) = \nu(f)$, where $f \in \Rat(X_{\overline{K}})^\times$ is any rational function defining $D$ around $c_X(\nu) \in X_{\overline{K}}$. Note that this definition does not depend on the choice of $f$ since $\nu(g) = 0$ for every $g \in \cO_{X_{\overline{K}},c_X(\nu)}^\times$. We obtain a map $ \Div(X_{\overline{K}}) \rightarrow \bZ^d$,
that extends to a map \[\nu \colon \Div(X_{\overline{K}})_{\bR} \rightarrow \bR^d\] after tensoring by $\bR$. By construction, we have 
\begin{itemize}
\item $\nu(\lambda D) = \lambda \nu(D)$ for any $D \in \Div(X)_\bR$ and any $\lambda \in \bR$,
\item  $\nu(D_1) + \nu(D_2) = \nu(D_1 + D_2)$ for any  $D_1,D_2 \in \Div(X)_\bR$,
\end{itemize}
Given a divisor $D \in \Div(X_{\overline{K}})_\bR$ and a global section $s =(D,\phi) \in \Gamma(X_{\overline{K}},D)$, we write $\nu(s) = \nu(\divi(s)) = \nu(D +(\phi))$.

Let $D \in \Div(X)_\bR$ be a big $\bR$-Cartier divisor, and let $V_{\bullet}(D) := \bigoplus_{n \in \bN} \Gamma(X,nD)$. Given a graded $K$-subalgebra $W_\bullet = \bigoplus_{n \in \bN} W_n \subseteq V_\bullet(D)$, we consider the subset $\Gamma_{\nu}(W_\bullet)$ of $\bR^d$ defined by 
\[\Gamma_{\nu}(W_\bullet) = \left\{\frac{\nu(s)}{n} \ | \ n \geq 1, \  s\in W_n \otimes_K \overline{K} \setminus \{0\} \right\}. \]

\begin{defi} The Okounkov body of $W_\bullet$ with respect to $\nu$ is the closure
\[\Delta_{\nu}(W_\bullet) := \overline{\Gamma_{\nu}(W_\bullet)} \]
of $\Gamma_{\nu}(W_\bullet)$ in $\bR^d$ for the Euclidean topology. When $W_\bullet = V_{\bullet}(D)$, $\Delta_{\nu}(D) := \Delta_{\nu}(V_\bullet(D))$ is called the Okounkov body of $D$ with respect to $\nu$. 
\end{defi}

We define the volume of $W_\bullet$ as 
 \[\vol(W_\bullet) = \limsup_{n \rightarrow \infty} \frac{\dim W_n}{n^d/d!}.\]
 Following \cite[Definition 1.1]{BC}, we say that $W_\bullet$ contains an ample series if
 \begin{itemize}
 \item $W_n \ne 0$ for every $n \gg 1$, and 
 \item there exist an integer $\ell \geq 1$ and an ample $\bQ$-Cartier divisor $A \in \Div(X)_\bQ$ such that $A \leq D$ and $H^0(X,n\ell A) \subseteq W_{n\ell}$ for every $n \geq 1$.
 \end{itemize}
 In that case, we have 
  \[\mu_{\bR^d}(\Delta_{\nu}(W_\bullet)) = \vol(W_\bullet)/d! =  \lim_{n \rightarrow \infty} \frac{\dim W_n}{n!}\]
  by \cite[Theorem 2.13]{LazMus}, where $\mu_{\bR^d}$ is the Lebesgue measure on $\bR^d$. Note that $V_\bullet(D)$ contains an ample series since $D$ is big. In particular, $\Delta_{\nu}(D) \subset \bR^d$ is a convex body and  $\vol(D) = d! \mu_{\bR^d}(\Delta_{\nu}(D))$.

\subsection{The arithmetic concave transform}\label{paragconcave} Let $\overline{D} = (D, (g_v)_{v \in \Sigma_K}) \in \Div(X)_\bR$ be an adelic $\bR$-Cartier divisor on $X$ such that $D$ is big. Let $\overline{\xi} = (0,(\xi_v)_{v \in \Sigma_K}) \in \widehat{\Div}(X)$ be such that $\widehat{\deg}(\overline{\xi}) = 1$ (see Example \ref{exampledeg}).  For any $n \in \bN\setminus \{0\}$ and $t \in \bR$, we denote by $V_n^t(\overline{D}) \subseteq \Gamma(X,nD)$ the $K$-linear subspace generated by  $\widehat{\Gamma}(X, n\overline{D} - t\overline{\xi})$. 
 We let $V^t_\bullet(\overline{D}) \subseteq V_\bullet(D)$ be the graded $K$-subalgebra defined by \[V^t_\bullet(\overline{D})= \bigoplus_{n\in \bN} V^{nt}_n(\overline{D}).\] 
\begin{defi} The concave transform of $\overline{D}$ with respect to $\nu$ is the function $G_{\overline{D}, \nu} \colon \Delta_{\nu}(D) \rightarrow \bR \cup \{-\infty\}$ given by 
\[G_{\overline{D},\nu}(\alpha) = \sup \{t \in \bR \ | \ \alpha \in \Delta_{\nu}(V_\bullet^t(\overline{D}))\}\]
for every $\alpha \in \Delta_{\nu}(D)$.
\end{defi}
By \cite[§~1.3]{BC}, $G_{\overline{D}, \nu}$ is an upper semi-continuous concave function, and it is continuous on the interior of $\Delta_{\nu}(D)$. 
\begin{lemma}\label{lemmainfG}  We have
\[\inf_{\alpha \in \Delta_{\nu}(D)} G_{\overline{D}, \nu}(\alpha)  = \sup \{t \in \bR \ | \ \vol(D) = \vol(V_{\bullet}^t(\overline{D}))\} \in \bR \cup \{-\infty\}.\]
In particular, $\inf_{\Delta_{\nu}(D)} G_{\overline{D}, \nu}$ does not depend on $\nu$.
\end{lemma}
\begin{proof} Let $\theta(\overline{D})$ be the supremum on the right hand-side. To prove that $\theta(\overline{D}) \geq \inf G_{\overline{D}, \nu}$, we assume that $\inf  G_{\overline{D}, \nu} > -\infty$ without loss of generality. For any real number  $t < \inf G_{\overline{D}, \nu}$, it follows from the definitions that  $\Delta_{\nu}(V_\bullet^t (\overline{D}))=\Delta_{\nu}(D)$. Moreover, $V_\bullet^t (\overline{D})$ contains an ample series by \cite[Lemma 1.6 and (1.8) page 1213]{BC}. Therefore $\vol (V_\bullet^t (\overline{D})) = \vol(D)$, and it follows that $\theta(\overline{D}) \geq \inf  G_{\overline{D}, \nu}$. Conversely, let $t \in \bR$ be a real number such that $\vol (V_\bullet^t (\overline{D})) = \vol(D) > 0$. By \cite[Lemma 1.6]{BC},  $V_\bullet^t (\overline{D})$ contains an ample series and therefore $\mu_{\bR^d}(\Delta_{\nu}(V_\bullet^t (\overline{D})))=\mu_{\bR^d}(\Delta_{\nu}(D))$. Since $\Delta_{\nu}(V_\bullet^t (\overline{D}))$ is a closed subset of $\Delta_{\nu}(D)$, we have $\Delta_{\nu}(V_\bullet^t (\overline{D}))=\Delta_{\nu}(D)$ and therefore $G_{\overline{D},\nu}(\alpha) \geq t$ for every $\alpha \in \Delta_{\nu}(D)$ by definition of $G_{\overline{D},\nu}$.
\end{proof}

We end this paragraph with a fundamental theorem of Boucksom and Chen relating the concave transform to the arithmetic volume.

\begin{theorem}[Boucksom--Chen]\label{thmBC} We have
\[\widehat{\vol}(\overline{D}) = (d+1)!\int_{\Delta_{\nu}(D)} \max \{0, G_{\overline{D},\nu}\}d\mu_{\bR^d}.\]
Moreover, if $\overline{D}$ is semi-positive  then
\[h_{\overline{D}}(X) \leq (d+1)!\int_{\Delta_{\nu}(D)} \max \{t, G_{\overline{D},\nu}\}d\mu_{\bR^d},\]
 for every $t \in \bR$, with equality if $\zeta_{\mathrm{abs}}(\overline{D}) \geq t$. 
\end{theorem}
\begin{proof} The first part of the theorem follows from \cite[Theorem 1.11]{BC} as in the proof  of \cite[Theorem 2.8]{BC}. In order to prove the second part, we assume that $\overline{D}$ is semi-positive. Let $\overline{\xi}= (0,(\xi_v)_{v \in \Sigma_K}) \in \widehat{\Div}(X)$ be the adelic $\bR$-Cartier divisor introduced at the beginning of §~\ref{paragconcave} and let $t \in \bR$ be a real number. We let $\overline{D}(t) = \overline{D} - t\overline{\xi}$.  By definition, we have $G_{\overline{D}(t),\nu} = G_{\overline{D},\nu}-t$. Moreover, $\overline{D}(t)$ is semi-positive. By Remark \ref{remaheight} \eqref{rematwistheight} and Theorem \ref{thmHodge}, we have 
\[h_{\overline{D}}(X) = h_{\overline{D}(t)}(X) +t(d+1) \deg_D(X) \leq \widehat{\vol}(\overline{D}(t)) +t(d+1) \deg_D(X),\]
with equality if $\zeta_{\mathrm{abs}}(\overline{D}) \geq t$ (since in that case $\zeta_{\mathrm{abs}}(\overline{D}(t)) = \zeta_{\mathrm{abs}}(\overline{D}) -t \geq 0$, so that $\overline{D}(t)$ is nef and $h_{\overline{D}(t)}(X) = \widehat{\vol}(\overline{D}(t))$ by Theorem \ref{thmHodge}). 
On the other hand, 
$\widehat{\vol}(\overline{D}(t)) = (d+1)!\int_{\Delta_{\nu}(D)} \max \{0, G_{\overline{D},\nu} -t \}d\mu_{\bR^d}$ and $\deg_D(X)= d! \mu_{\bR^d}(\Delta_\nu(D))$. Therefore
\begin{align*}
h_{\overline{D}}(X) & \leq  (d+1)!\int_{\Delta_{\nu}(D)} ( \max \{0, G_{\overline{D},\nu} -t \} + t)d\mu_{\bR^d}\\
&  = (d+1)!\int_{\Delta_{\nu}(D)} \max \{t, G_{\overline{D},\nu} \}d\mu_{\bR^d},
\end{align*}
with equality if $\zeta_{\mathrm{abs}}(\overline{D}) \geq t$. 
\end{proof}

\subsection{The generalized arithmetic concave transform}\label{paraggeneralconc} In the previous paragraph, we defined the arithmetic concave transform for adelic $\bR$-Cartier divisors $\overline{D} = (D, (g_v)_{v\in \Sigma_K})$ with $D$ big. Our goal in this subsection is to extend this construction to the case where $D$ is pseudo-effective (Definition \ref{defConcPseff}). We first need to define Okounkov bodies of pseudo-effective divisors.
\begin{defi}\label{defiGeomNOpseff} Let $D \in \Div(X)_{\bR}$. If $D$ is pseudo-effective, the Okounkov body of $D$ with respect to $\nu$ is 
\[\Delta_\nu(D) = \bigcap_{A \text{ ample}} \Delta_\nu(D+A),\]
where the intersection is over all ample $\bR$-Cartier divisors $\overline{A}$ on $X$. If $D$ is not pseudo-effective, we put $\Delta_\nu(D) = \emptyset$.
\end{defi}
 When $D$ is big, the Okounkov body $\Delta_\nu(D)$ of Definition \ref{defiGeomNOpseff} coincides with the one defined in §~\ref{paragNO} (see \cite[Proposition 4.4]{BoucksomOkounkov}). 
 
\begin{lemma}\label{lemmaampleNO} For any $\overline{A} \in \widehat{\Div}(X)_\bR$ generated by strictly small $\bR$-sections, we have $0_{\bR^d} \in \Delta_{\nu}(V_\bullet^0(\overline{A}))$. In particular, $G_{\overline{A},\nu}(0_{\bR^d}) \geq 0$.
\end{lemma}

\begin{proof} Let us first assume that $\overline{A} \in \widehat{\Div}(X)$. By \cite[Lemma 3.17]{BMPS}, there exists an integer $n \geq 1$ and a small global section $s \in \widehat{\Gamma}(X,n\overline{A})^\times \subseteq V_ n^0(\overline{A})$ such that $c_X(\nu) \notin \Supp(\divi(s))_{\overline{K}}$. Therefore, if $f \in \Rat(X_{\overline{K}})^\times$ defines $\divi(s)_{\overline{K}}$ around $c_X(\nu)$, we have $f = \cO_{X_{\overline{K}},c_X(\nu)}^\times$. Hence  $0_{\bR^{d}} = \nu(f)/n = \nu(s)/n\in \Delta_{\nu}(V_\bullet^0(\overline{A}))$.

 We now consider the general case. As shown in \cite[Proof of Proposition 6.5.3]{MoriwakiZariski}, there exist adelic $\bR$-Cartier divisors $\overline{A}_1, \ldots, \overline{A}_r$ satisfying the following conditions:
\begin{itemize}
\item for every $i \in \{1, \ldots, r \}$, we have $1 \in \widehat{\Gamma}(X,\overline{A}_i)^\times$;
\item for every $\varepsilon > 0$, there exist positive real numbers $\delta_1, \ldots, \delta_r \in (0,\varepsilon)$ such that $\overline{A}  - \sum_{i=1}^r \delta_i \overline{A}_i \in \widehat{\Div}(X)_\bQ$ is generated by strictly small $\bR$-sections.
\end{itemize}
 Let $\varepsilon > 0$ and $\delta_1, \ldots, \delta_r \in (0,\varepsilon)$ be as above. By construction, we have
\[\nu(V_n^0(\overline{A}  - \sum_{i=1}^r \delta_i \overline{A}_i)) +n\sum_{i=1}^r \delta_i  \nu(A_i) \subseteq \nu(V_n^0(\overline{A}))\]
for any integer $n \geq 1$, and therefore 
\[\Delta_{\nu}^0(\overline{A} - \sum_{i=1}^r \delta_i \overline{A}_i) + \sum_{i=1}^r \delta_i \nu(A_i)  \subseteq \Delta_{\nu}^0(\overline{A}).\]
By the above we have $0_{\bR^d} \in \Delta_{\nu}^0(\overline{A} - \sum_{i=1}^r \delta_i \overline{A}_i)$, thus $\sum_{i=1}^r \delta_i \nu(A_i)  \subseteq \Delta_{\nu}^0(\overline{A})$. By letting $\varepsilon$ tend to zero, it follows that $0_{\bR^d} \in \Delta_{\nu}^0(\overline{A})$.
\end{proof}

 \begin{lemma}\label{lemmaConcinf} Let $\overline{D} = (D, (g_v)_{v\in \Sigma_K}) \in \widehat{\Div}(X)$, and assume that $D$ is big. For every $\alpha  \in \Delta_\nu(D)$, we have 
 \[G_{\overline{D},\nu}(\alpha) = \inf_{\overline{A} \text{ ample }} G_{\overline{D} + \overline{A}}(\alpha),\]
 where the infimum is over all ample adelic $\bR$-Cartier divisors $\overline{A}$ on $X$. 
\end{lemma}
\begin{proof} We let $\widetilde{G}_{\overline{D},\nu} = \inf_{\overline{A} \text{ ample }} G_{\overline{D} + \overline{A}}$ be the function on the right hand-side.   Let $\overline{A}$ be an ample adelic $\bR$-Cartier divisor. For any $t \in \bR$ and for any integer $n \geq1$, 
$\nu(V_n^{nt}(\overline{D})) +\nu(V_n^{0}(\overline{A}))$ is a subset of
\[ \left\{ \nu(\divi(s) + \divi(s')) \ | \ s \in V_n^{nt}(\overline{D}) \setminus \{0\}, \ s' \in V_n^{0}(\overline{A}) \setminus \{0\} \right\}
  \subseteq \nu(V_n^{nt}(\overline{D}+\overline{A})).
\] 
This implies that ${\Delta}_{\nu}(V_{\bullet}^t(\overline{D})) + {\Delta}_{\nu}(V_{\bullet}^{0}(\overline{A})) \subseteq {\Delta}_{\nu}(V_{\bullet}^{t}(\overline{D}+\overline{A}))$, and therefore \[{\Delta}_{\nu}(V_{\bullet}^t(\overline{D}))\subseteq {\Delta}_{\nu}(V_{\bullet}^{t}(\overline{D}+\overline{A}))\]
 by Lemma \ref{lemmaampleNO}. It follows that $ G_{\overline{D},\nu}  \leq \widetilde{G}_{\overline{D},\nu}$. To prove the converse inequality, we argue by contradiction and we assume that $G_{\overline{D},\nu} < \widetilde{G}_{\overline{D},\nu}$. Using that $\widetilde{G}_{\overline{D},\nu}$ is concave, that $G_{\overline{D},\nu}$ is upper semi-continuous on  $\Delta_\nu(D)$, and that $G_{\overline{D},\nu}$  is continuous on the interior of $\mathrm{int}(\Delta_\nu(D))$, it is easy to see that there exists $\alpha \in \mathrm{int}(\Delta_\nu(D))$ such that  $G_{\overline{D},\nu}(\alpha) < \widetilde{G}_{\overline{D},\nu}(\alpha)$.  Let $\overline{\xi} \in \widehat{\Div}(X)$ be the adelic Cartier divisor defined in §~\ref{paragconcave}. Then  $G_{\overline{D} +\varepsilon\overline{\xi}, \nu} =  G_{\overline{D},\nu}  + \varepsilon$ and $\widetilde{G}_{\overline{D} +\varepsilon\overline{\xi},\nu} =  \widetilde{G}_{\overline{D}, \nu}  + \varepsilon$ for any $\varepsilon \in \bR$ by definition. Let $\varepsilon$ be a real number such that $G_{\overline{D} +\varepsilon\overline{\xi}, \nu} (\alpha) > 0$. Since $\alpha \in \mathrm{int}(\Delta_\nu(D))$ and  $0 < G_{\overline{D} +\varepsilon\overline{\xi}, \nu}(\alpha) <  \widetilde{G}_{\overline{D} +\varepsilon\overline{\xi}, \nu}(\alpha)$, there exist a real number $\lambda > 0$ and an open subset $U \subseteq \mathrm{int}(\Delta_\nu(D))$ with positive Lebesgue measure such that $0 < G_{\overline{D} +\varepsilon\overline{\xi}, \nu}(\beta) <  \widetilde{G}_{\overline{D} +\varepsilon\overline{\xi}, \nu}(\beta) - \lambda$ for every $\beta \in U$ (note that $ G_{\overline{D} +\varepsilon\overline{\xi}, \nu}$ and $ \widetilde{G}_{\overline{D} +\varepsilon\overline{\xi}, \nu}$ are concave, hence continuous on $\mathrm{int}(\Delta_\nu(D)$).  Theorem \ref{thmBC} implies that there exists a real number $\lambda' > 0$ such that 
 \[\widehat{\vol}(\overline{D} +\varepsilon\overline{\xi})+ \lambda' < \widehat{\vol}(\overline{D} +\varepsilon\overline{\xi} + \overline{A}) \]
for any ample adelic $\bR$-Cartier divisor $\overline{A}$ on $X$. This contradicts the continuity property  \eqref{eqcontvol} of $\widehat{\vol}$, and finishes the proof. 
\end{proof}

\begin{rema} It follows from the proof that Lemma \ref{lemmaConcinf} remains valid when the infimum is taken over all adelic $\bR$-Cartier divisors generated by strictly small $\bR$-sections. We will not use this fact in this text.
\end{rema}
Lemma \ref{lemmaConcinf} allows us to extend the definition of the arithmetic concave transform as follows. 
\begin{defi}\label{defConcPseff} Let $\overline{D} = (D, (g_v)_{v\in \Sigma_K})$ be an adelic $\bR$-Cartier divisor on $X$ such that $D$ is pseudo-effective. The arithmetic concave transform of $\overline{D}$ with respect to $\nu$ is the concave function  $G_{\overline{D}, \nu} \colon \Delta_\nu(D) \rightarrow \bR \cup \{-\infty\}$
defined by 
\[G_{\overline{D},\nu}(\alpha) = \inf_{\overline{A} \text{ ample}} G_{\overline{D}+\overline{A},\nu}(\alpha),\]
for every $\alpha \in \Delta_\nu(D)$, where the infimum is over all ample adelic $\bR$-Cartier divisors $\overline{A}$  on $X$. 
\end{defi}

We end this paragraph with an alternative description of the arithmetic concave transform.

\begin{lemma}\label{lemmaConcalt} Let $\overline{D} = (D, (g_v)_{v\in \Sigma_K})$ be an adelic $\bR$-Cartier divisor on $X$ such that $D$ is pseudo-effective, and let $\overline{A}_0 \in \widehat{\Div}(X)_{\bR}$ be generated by strictly small $\bR$-sections. For every $\alpha \in \Delta_\nu(D)$, we have 
\[G_{\overline{D}, \nu}(\alpha) = \lim_{m \rightarrow \infty} G_{\overline{D} + \frac{1}{m}\overline{A}_0,\nu}(\alpha).\]
Moreover, for every $\alpha \in \Delta_\nu(D)$ we have 
\[G_{\overline{D}, \nu}(\alpha) = \inf_{\substack{\overline{A} \text{ ample } \\ \overline{D} + \overline{A} \in \widehat{\Div}(X)_\bQ}} G_{\overline{D} + \overline{A}, \nu}(\alpha),\]
where the infimum is over all ample adelic $\bR$-Cartier divisors on $X$ such that $\overline{D} + \overline{A} \in \widehat{\Div}(X)_\bQ$ is an adelic $\bQ$-Cartier divisor.
\end{lemma}

\begin{proof} Arguing as in the proof of Lemma \ref{lemmaConcinf}, we have 
\[\Delta_\nu^t(\overline{D}) \subseteq \Delta_\nu^t(\overline{D} + \frac{1}{m+1}\overline{A}_0) \subseteq \Delta_\nu^t(\overline{D} + \frac{1}{m}\overline{A}_0)\]
for every $t \in \bR$ and $m \in \bN \setminus \{0\}$. It follows that 
\[G_{\overline{D}, \nu}(\alpha) \leq  G_{\overline{D} + \frac{1}{m+1}\overline{A}_0,\nu}(\alpha) \leq G_{\overline{D} + \frac{1}{m}\overline{A}_0,\nu}(\alpha)\]
for every $\alpha \in \Delta_\nu(D)$. Let $\overline{A} \in \widehat{\Div}(X)_\bR$ be ample. Then for $m \geq 1$ sufficiently large, $\overline{A} - \frac{1}{m} \overline{A}_0$ is generated by strictly small $\bR$-sections. As above, it follows that $G_{\overline{D} + \frac{1}{m}\overline{A}_0,\nu}(\alpha) \leq G_{\overline{D} +\overline{A},\nu}(\alpha)$. Taking the infimum on $\overline{A}$, we obtain that
\[G_{\overline{D}, \nu}(\alpha) \leq  \lim_{m \rightarrow \infty} G_{\overline{D} + \frac{1}{m}\overline{A}_0,\nu}(\alpha) \leq G_{\overline{D},\nu}(\alpha) = \inf_{\overline{A} \text{ ample}} G_{\overline{D}+\overline{A},\nu}(\alpha).\]
This proves the first part of the lemma. The second one can be proved with the same arguments, after observing that for every ample $\overline{A} \in \widehat{\Div}(X)_\bR$ there exists an ample $\overline{A}' \in \widehat{\Div}(X)_\bR$ such that $\overline{D} + \overline{A}' \in \widehat{\Div}(X)_\bQ$ and $\overline{A} - \overline{A}'$ is generated by strictly small $\bR$-sections.
\end{proof}

\subsection{Arithmetic Okounkov body}\label{paragarithmNO}  

\begin{defi}\label{defiarithmNO} Let $\overline{D} = (D,(g_v)_{v\in \Sigma_K}) \in \widehat{\Div}(X)_\bR$. If $D$ is pseudo-effective, we define the arithmetic Okounkov body of $\overline{D}$ as
\[ \widehat{\Delta}_{\nu}(\overline{D}) =  \{ (\alpha, t) \in \Delta_{\nu}(D)\times \bR_+ \ | \ G_{\overline{D},\nu}(\alpha) \geq t  \}.\]
 If $D$ is not pseudo-effective, we put $ \widehat{\Delta}_{\nu}(\overline{D}) =  \emptyset$.
\end{defi}
It follows from the definitions that for any $\overline{D} =  (D,(g_v)_{v\in \Sigma_K}) \in \widehat{\Div}(X)_\bR$, we have 
\[ \widehat{\Delta}_{\nu}(\overline{D}) = \bigcap_{\overline{A}\ \text{ample}} \widehat{\Delta}_{\nu}(\overline{D} +\overline{A}),\]
where the  intersection is over all ample adelic $\bR$-Cartier divisors  $\overline{A}$ on $X$.  We also have the following useful reformulation of Lemma \ref{lemmaConcalt} in terms of arithmetic Okounkov bodies.
 \begin{lemma}\label{lemmacapNO}  Let $\overline{D} = (D,(g_v)_{v\in \Sigma_K}) \in \widehat{\Div}(X)_\bR$, and let $\overline{A}_0 \in \widehat{\Div}(X)_\bR$ be generated by strictly small $\bR$-sections. We have 
 \[ \widehat{\Delta}_{\nu}(\overline{D}) = \bigcap_{m \in \bN \setminus \{0\}} \widehat{\Delta}_{\nu}(\overline{D} + \frac{1}{m}\overline{A}_0) = \bigcap_{\substack{\overline{A}\ \text{ample} \\ \overline{D}+\overline{A} \in \widehat{\Div}(X)_\bQ}} \widehat{\Delta}_{\nu}(\overline{D} +\overline{A}),\]
where the last intersection is over all ample adelic $\bR$-Cartier divisors  $\overline{A}$ on $X$ such that $\overline{D}+\overline{A} \in \widehat{\Div}(X)_\bQ$. 
\end{lemma}

\begin{proof} It is well-known and not hard to see that the geometric Okounkov body satisfies
\[\Delta_\nu(D) = \bigcap_{m \in \bN \setminus \{0\}} \Delta_\nu(D+\frac{1}{m}A_0)= \bigcap_{\substack{A\ \text{ample} \\ D+A\in {\Div}(X)_\bQ}}  {\Delta}_{\nu}(D+A).\]
Therefore the result follows directly from the definitions and  Lemma \ref{lemmaConcalt}. 
\end{proof}
 In the case where $D$ is big, Lemma \ref{lemmaConcinf} implies that $ \widehat{\Delta}_{\nu}(\overline{D})$ coincides with the arithmetic Okounkov body introduced in  \cite[Definition 2.7]{BC}. In particular, we have the following.
 \begin{lemma}\label{lemmadefOk2}  Let $\overline{D} = (D,(g_v)_{v\in \Sigma_K}) \in \widehat{\Div}(X)_\bR$. If $D$ is big, then 
 \[ \widehat{\Delta}_{\nu}(\overline{D})   = \{ (\alpha, t) \in \Delta_{\nu}(D)\times \bR_+ \ | \ \alpha \in \Delta_\nu(V_\bullet^t(\overline{D}))  \}\]
and 
 \[(d+1)!\mu_{\bR^{d+1}}(\widehat{\Delta}_{\nu}(\overline{D}))= \widehat{\vol}(\overline{D}),\]
where $\mu_{\bR^{d+1}}$ is the Lebesgue measure on $\bR^{d+1}$.
 \end{lemma}
 \begin{proof} The first equality follows from \cite[Remark 1.10]{BC}  (see \cite[(3.20.1)]{MLKR} for details). The second one is the first part of Theorem \ref{thmBC}.
 \end{proof}

We end this section with a characterization of pseudo-effective adelic $\bR$-Cartier divisors.
\begin{lemma}\label{lemmaNOpseff} Let $\overline{D} = (D,(g_v)_{v\in \Sigma_K}) \in \widehat{\Div}(X)_\bR$. Then $\overline{D}$ is pseudo-effective if and only if 
$\widehat{\Delta}_{\nu}(\overline{D}) \ne \emptyset$.
\end{lemma}

\begin{proof} Let $\overline{A}_0 \in \widehat{\Div}(X)_\bR$ be an ample adelic $\bR$-Cartier divisor on $X$. Assume that $\overline{D}$ is pseudo-effective. Then $\overline{D} + \frac{1}{m}\overline{A}_0$ is big for any integer $m \geq 1$. It follows that $\widehat{\Delta}_{\nu}(\overline{D}) $ is the intersection of a decreasing nested sequence of convex bodies, and therefore it is non-empty by Cantor's intersection theorem. Conversely, assume that $\widehat{\Delta}_{\nu}(\overline{D}) \ne \emptyset$. Then $D$ is pseudo-effective and $\widehat{\Delta}_{\nu}(\overline{D}+\frac{1}{m}\overline{A}_0) \ne \emptyset$ for any integer  $m \geq 1$. By definition, there exists $\alpha_m \in \Delta_{\nu}(D+ \frac{1}{m}A_0)$ such that $G_{\overline{D}+ \overline{A}_0/m,\nu}(\alpha_m) \geq 0$. 
 By  \cite[Remark 5.7 and Lemma 6.3]{BaEM},
$\overline{D}+\frac{1}{m}\overline{A}_0$ is pseudo-effective for any $m\geq 1$. Therefore $\overline{D}$ is pseudo-effective.
\end{proof}

\section{Positivity via arithmetic Okounkov bodies}\label{sectionpositivNO}
 We prove Theorem \ref{thmkeyintro} in § \ref{paragpropkey}. In § \ref{paragnefNO}, we  prove Corollary \ref{corominabsintro} and a refinement of Theorem \ref{thmnefintro} characterizing arithmetic nefness (Corollaries \ref{coronef} and \ref{corominabsG}). We then give several characterizations of arithmetic ampleness in terms of arithmetic Okounkov bodies in §~\ref{paragampleNO}, where we prove variants of Theorem \ref{thmampleintro} (see Corollaries \ref{thmadmample} and \ref{thminfample}).

\subsection{Height of points and arithmetic concave transform}\label{paragpropkey} 
 We shall establish Theorem \ref{thmkeyintro} as a  consequence of Proposition \ref{propCauchy}.

\begin{theorem}\label{thmkey}  Let $\overline{D} =(D,(g_v)_{v\in \Sigma_K}) \in \widehat{\Div}(X)_{\bR}$ and let  $\nu \in \mathcal{V}(X_{\overline{K}})$ be a valuation of maximal rank on $\Rat(X_{\overline{K}})$ centered at a point $x \in X_{\overline{K}}$.  If $D$ is nef, then $0_{\bR^d} \in \Delta_{\nu}(D)$ and 
  \[\widehat{h}_{\overline{D}}(x) \geq G_{\overline{D}, \nu}(0_{\bR^d}).\] 
\end{theorem}

\begin{proof} Since $D$ is nef, we have $0_{\bR^d} \in \Delta_{\nu}(D)$ by \cite[(20) page 1059-30]{BoucksomOkounkov}. We shall divide the proof of the inequality $\widehat{h}_{\overline{D}}(x) \geq t_0:= G_{\overline{D}, \nu}(0_{\bR^d})$ into three steps. 

\textbf{Step 1.} We suppose that $x \in X_{\overline{K}}$ is a regular point and that $D \in \Div(X)_\bQ$ is an ample $\bQ$-Cartier divisor.  
 We have  $0_{\bR^{d}} \in \Delta_{\nu}(V_\bullet^{t_0}(\overline{D}))$ by Lemma \ref{lemmadefOk2}, and in particular $V_\bullet^{t_0}(\overline{D}) \ne \{0\}$. Let us prove the following claim.
  \begin{claim}\label{claimexists0} For any $\sigma > 0$, there exist  an integer $m \geq 1$  and a non-zero global section $s \in \Gamma(X, m\overline{D})^\times$ such that $h_{m\overline{D}}(s) < -mt_0$ and $\ord_x s <  m\sigma$.
  \end{claim}
\begin{proof}
 Let $m \geq 1$ be an integer and let $\widetilde{s} \in V_m^{mt_0}(\overline{D}) \otimes_K \overline{K}\setminus \{0\}$. By definition, there exist a family $s_1, \ldots, s_{r_m} \in \Gamma(X, m\overline{D})^\times$ and $\lambda_1, \ldots, \lambda_{r_m} \in \overline{K}^\times$ such that $h_{m\overline{D}}(s_i) < -mt_0$ for every $i \in \{1, \ldots,r_m\}$ and $\widetilde{s} = \sum_{i = 1}^{r_m}\lambda_i s_i$. Assume by contradiction that the conclusion of the claim does not hold and let $i \in \{1, \ldots, r_m\}$. Then $ \ord_x s_i \geq m\sigma$. Let $z_1, \ldots, z_d$ be a system of parameters at $x$ and let $\bm{\beta}_x =  \min_{1 \leq j \leq d} \nu(z_j) \in \bZ^d$. Since $x=c_X(\nu)$, we have $\bm{\beta}_x >_{\mathrm{lex}}  0_{\bZ^d}$. 
  By definition of $\ord_x(s_i)$,  $\divi(s_i)$ is locally defined by a function $f \in \mathfrak{m}_{X_{\overline{K}},x}^{m\sigma}\cO_{X_{\overline{K}},x}$ around $x$, and therefore $\nu(s_i) = \nu(f) \geq_{\mathrm{lex}} m\sigma \bm{\beta}_x$. 
  Hence 
\[\nu(\widetilde{s}) \geq_{\mathrm{lex}} \min_{1 \leq i \leq r_m} \nu (s_i) \geq_{\mathrm{lex}} m\sigma \bm{\beta}_x.\]
 It follows that
 \[ \Gamma_{\nu} (V_\bullet^{t_0}(\overline{D})) \subseteq \{ \bm{\alpha} \in \bQ^d \ | \ \bm{\alpha} \geq_{\mathrm{lex}} \sigma \bm{\beta}_x\},\]
 and thus $\bm{\alpha} \geq_{\mathrm{lex}} \sigma \bm{\beta}_x >_{\mathrm{lex}} 0_{\bR^d}$ for every  $\bm{\alpha} \in \Delta_{\nu} (V_\bullet^{t_0}(\overline{D}))$.
 This is a contradiction since $0_{\bR^d} \in \Delta_{\nu} (V_\bullet^{t_0}(\overline{D}))$.
\end{proof}    
Let $\varepsilon$ and $\sigma$ be positive real numbers. By Claim \ref{claimexists0}, there exist $m \in \bN\setminus\{0\}$ and $s \in \Gamma(X, m\overline{D})^\times$  with $h_{m\overline{D}}(s) <-mt_0$ and $\ord_x(s) < m\sigma$.  By Proposition \ref{propCauchy}, we have 
 \[\widehat{h}_{\overline{D}} (x) +\varepsilon \geq - \frac{1}{m}(h_{m \overline{D}}(s) + \rho(\overline{D},x, \varepsilon) \ord_x s)> t_0 -\sigma |\rho(\overline{D},x, \varepsilon)|,\]
 where $\rho(\overline{D},x, \varepsilon) \in \bR$ depends only on $\overline{D}$, $x$ and $\varepsilon$. By letting $\sigma$ tend to zero, we obtain $\widehat{h}_{\overline{D}}(x) \geq t_0 - \varepsilon$, and therefore $\widehat{h}_{\overline{D}}(x) \geq t_0$. That proves the theorem in the case where $x$ is regular and $D$ is an ample  $\bQ$-Cartier divisor. 
 
 \textbf{Step 2.} We continue to suppose that $D \in \Div(X)$ is an ample $\bQ$-Cartier divisor, but we allow $x \in X_{\overline{K}}$ to be a singular point.  
   By de Jong's alteration Theorem \cite[Theorem 4.1]{dJong},  
 there exist a finite extension $K'$ of $K$ and a smooth geometrically integral projective variety $X'$ over $K'$, together with a surjective generically finite morphism $\mu \colon X' \rightarrow X_{K'}$. We denote by $\overline{D}' \in \widehat{\Div}(X')_\bR$ the pullback of $\overline{D}$ to $X'$, and we let $\nu' \in \mathcal{V}(X'_{\overline{K}})$ be a valuation of maximal rank extending $\nu$ (such a valuation exists by \cite[Chap.\ VI, § 6, Theorem 11, page 26]{ZS}). The center $x' = c_{X'}(\nu')$ of $\nu'$ satisfies $\mu(x') = x$, and in particular $\widehat{h}_{\overline{D}'}(x') = \widehat{h}_{\overline{D}}(x)$.   
 We have $0_{\bR^d} \in \Delta_{\nu}(V_{\bullet}^{t_0}(\overline{D}))$ by Lemma \ref{lemmadefOk2} (recall that $t_0 = G_{\overline{D},\nu}(0_{\bR^d})$). It follows that for any positive integer $n$, there exist $m_n \in \bN \setminus \{0\}$ and 
 $s_n \in V_{m_n}^{m_nt_0}(\overline{D}) \setminus \{0\}$ such that $\alpha_n := \nu(s_n)/m_n$ satisfies $|{\alpha}_n| < 1/n$. By construction, $\mu^*s_n \in V_{m_n}^{m_nt_0}(\overline{D}')$ and $\nu'(\mu^*s_n) = \nu(s_n)$ for every $n > 0$. Therefore $\alpha_n \in \Delta_{\nu'}(V_{\bullet}^{t_0}(\overline{D}'))$. Letting $n$ tend to infinity, it follows that $0_{\bR^d} \in \Delta_{\nu'}(V_{\bullet}^{t_0}(\overline{D}'))$ and thus $G_{\overline{D}', \nu'}(0_{\bR^d}) \geq t_0$. By Step 1, we have  
 \[\widehat{h}_{\overline{D}}(x) = \widehat{h}_{\overline{D}'}(x') \geq G_{\overline{D}', \nu'}(0_{\bR^d}) \geq t_0.\]
 
\textbf{Step 3.} Finally, we consider the general case. There exist ample adelic $\bR$-Cartier divisors $\overline{A}_1, \ldots, \overline{A}_n$ such that for any real number $\delta > 0$, there exists $\mathbf{t} = (t_1, \ldots, t_d) \in \bR_+^n$ such that $|\mathbf{t}| < \delta$ and 
  \[\overline{D}_{\mathbf{t}} =\overline{D} + \sum_{i=1}^n t_i \overline{A}_i \in \widehat{\Div}(X)_\bQ.\]
  Let $\delta > 0$ be a real number and let  $\mathbf{t} = (t_1, \ldots, t_n) \in \bR_+^n$ be a $n$-tuple such that $|\mathbf{t}| < \delta$ and 
 $\overline{D}_{\mathbf{t}}  \in \widehat{\Div}(X)_\bQ$. Since $D$ is nef, the underlying divisor ${D}_{\mathbf{t}}$ is ample. By Step 2, we have \[h_{\overline{D}_{\mathbf{t}}}(x) \geq G_{\overline{D}_{\mathbf{t}},\nu}(0_{\bR^d}) \geq  G_{\overline{D},\nu}(0_{\bR^d}),\]
 where the second inequality follows from the definition  of $G_{\overline{D},\nu}$ (Definition \ref{defConcPseff}). On the other hand, $h_{\overline{A}_i}(x) > 0$ for every $i\in \{1, \ldots, n\}$, hence 
 \[ \widehat{h}_{\overline{D}}(x) + \delta \sum_{i = 1}^n \widehat{h}_{\overline{A}_i}(x) \geq \widehat{h}_{\overline{D}_{\mathbf{t}}}(x) \geq G_{\overline{D},\nu}(0_{\bR^d}).\]
 Letting $\delta$ tend to zero, we conclude that $\widehat{h}_{\overline{D}}(x)\geq G_{\overline{D},\nu}(0_{\bR^d})$.
\end{proof}

\subsection{Characterization of arithmetic nefness}\label{paragnefNO}
 As a direct consequence of Theorem \ref{thmkey}, we have the following refinement of Theorem \ref{thmnefintro} (see Remark \ref{remacharnefconcOk} below).  
 \begin{coro}\label{coronef} Let $\overline{D} = (D,(g_v)_{v \in \Sigma_K}) \in \widehat{\Div}(X)_\bR$ be semi-positive and let $\mathcal{V}_0 \subseteq \mathcal{V}(X_{\overline{K}})$ be a subset such that for every closed point $x \in X_{\overline{K}}$, there exists $\nu \in \mathcal{V}_0$ with center $c_X(\nu) = x$. The following conditions are equivalent:
\begin{enumerate}
\item\label{critnef}  $\overline{D}$ is nef;
\item\label{critcoroGall} for every $\nu \in \mathcal{V}_0$, we have $G_{\overline{D}, \nu}(\alpha) \geq 0$ for all $\alpha \in \Delta_{\nu}(D)$;
\item\label{critcoroGzero} for every $\nu \in \mathcal{V}_0$, we have $G_{\overline{D}, \nu}(0_{\bR^d}) \geq 0$;
\item\label{critcoroGallexists}  for every closed point $x \in X_{\overline{K}}$, there exists a valuation of maximal rank $\nu \in \mathcal{V}_0$ centered at $x$ such that  $G_{\overline{D}, \nu}(\alpha) \geq 0$ for all $\alpha \in \Delta_{\nu}(D)$;
\item\label{critcoroGzeroexists} for every closed point $x \in X_{\overline{K}}$, there exists a valuation of maximal rank $\nu \in \mathcal{V}_0$ centered at $x$ such that $G_{\overline{D}, \nu}(0_{\bR^d}) \geq 0$.
\setcounter{count}{\value{enumi}}
\end{enumerate}
Moreover, if $D$ is big then the above conditions are equivalent to the following:
\begin{enumerate} \setcounter{enumi}{\value{count}}
\item\label{critcoroGexists} there exists a valuation of maximal rank $\nu \in \mathcal{V}(X_{\overline{K}})$ such that $G_{\overline{D}, \nu}(\alpha) \geq 0$ for all $\alpha \in \Delta_{\nu}(D)$.
\end{enumerate}
\end{coro}

\begin{proof}
 We first prove that $\eqref{critnef} \Rightarrow \eqref{critcoroGall}$. Assume that $\overline{D}$ is nef and let $\overline{A}$ be an ample adelic $\bR$-Cartier divisor such that $\overline{D} + \overline{A} \in \widehat{\Div}(X)_\bQ$. Since $\overline{D}$ is nef, $\overline{D} + \overline{A}$ is ample. By Theorem \ref{thmNM} and \cite[Theorem 7.4.1]{CMadelic}, we have $V^0_\bullet(\overline{D} + \overline{A}) = V_\bullet(D+A)$. Therefore $G_{\overline{D} + \overline{A}, \nu} \geq 0$ for every $\nu \in \mathcal{V}_0$ by Lemma \ref{lemmainfG}, and \eqref{critcoroGall} holds by  Lemma \ref{lemmaConcalt}. The implications  $\eqref{critcoroGall} \Rightarrow \eqref{critcoroGallexists} \implies \eqref{critcoroGzeroexists} $ and  $\eqref{critcoroGall} \Rightarrow \eqref{critcoroGzero} \implies \eqref{critcoroGzeroexists} $ are trivial. If $\eqref{critcoroGzeroexists}$ holds, then $\widehat{h}_{\overline{D}}(x) \geq 0$ for every closed point $x \in X_{\overline{K}}$, and therefore $\overline{D}$ is nef.  Finally, we note that when $D$ is big then $\eqref{critcoroGall} \Leftrightarrow \eqref{critcoroGexists}$ by Lemma \ref{lemmainfG}.
\end{proof}

\begin{rema}\label{remacharnefconcOk} Theorem \ref{thmnefintro} in the introduction is nothing but a translation in terms of arithmetic Okounkov bodies of the equivalences $\eqref{critnef} \Leftrightarrow \eqref{critcoroGzero} \Leftrightarrow \eqref{critcoroGzeroexists}$ given by Corollary \ref{coronef} in the case $\mathcal{V}_0 = \mathcal{V}(X_{\overline{K}})$. 
\end{rema}

We end this subsection with a slight generalization of Corollary \ref{corominabsintro}, that gives new descriptions of the absolute minimum $\zeta_{\mathrm{abs}}(\overline{D}) = \inf_{x \in X(\overline{K})} \widehat{h}_{\overline{D}}(x)$. 

\begin{coro}\label{corominabsG} Let $\overline{D} = (D,(g_v)_{v \in \Sigma_K}) \in \widehat{\Div}(X)_\bR$ be semi-positive and let $\mathcal{V}_0 \subseteq \mathcal{V}(X_{\overline{K}})$ be a subset such that for every closed point $x \in X_{\overline{K}}$, there exists $\nu \in \mathcal{V}_0$ with center $c_X(\nu) = x$. Then 
\[ \zeta_{\mathrm{abs}}(\overline{D}) = \inf_{\nu \in \mathcal{V}_0} G_{\overline{D}, \nu}(0_{\bR^d}).\] 
Moreover, if $D$ is big then  $\zeta_{\mathrm{abs}}(\overline{D})=\inf_{\alpha \in \Delta_{\nu}(D)} G_{\overline{D}, \nu}(\alpha)$
for any $\nu \in \mathcal{V}(X_{\overline{K}})$. 
\end{coro}

\begin{proof}   Let $\overline{\xi} = (0, (\xi_v)_{v \in \Sigma_K})$ be the adelic Cartier divisor defined in §~\ref{paragconcave}. 
  For $t \in \bR$, we let $\overline{D}(t) = \overline{D} - t\overline{\xi}$. Note that since $\overline{D}$ is semi-positive, so is $\overline{D}(t)$. By definition, $\overline{D}(t)$ is nef if and only if $\zeta_{\mathrm{abs}}(\overline{D})-t = \zeta_{\mathrm{abs}}(\overline{D}(t)) \geq 0$. Moreover, by construction we have $G_{\overline{D}(t),\nu}(\alpha) = G_{\overline{D},\nu}(\alpha) -t$ for any $\nu \in \mathcal{V}(X_{\overline{K}})$ and $\alpha \in \Delta_{\nu}(D)$. By Corollary \ref{coronef}, we have 
  \[\zeta_{\mathrm{abs}}(\overline{D}) \geq t \Longleftrightarrow  \inf_{\nu \in \mathcal{V}_0} G_{\overline{D}, \nu}(0_{\bR^d}) \geq t,\] 
 and if moreover $D$ is big then
\[\zeta_{\mathrm{abs}}(\overline{D}) \geq t \Longleftrightarrow  \inf_{\alpha \in \Delta_\nu(D)} G_{\overline{D}, \nu}(\alpha) \geq t\] 
 for any $\nu \in \mathcal{V}(X_{\overline{K}})$. The corollary follows.
 \end{proof}

\subsection{Characterizations of arithmetic ampleness}\label{paragampleNO}
Let us begin with a direct consequence of Corollary \ref{corominabsG}.
 
 \begin{coro}\label{coroample}  Let $\overline{D} = (D,(g_v)_{v \in \Sigma_K}) \in \widehat{\Div}(X)_\bR$ be a semi-positive adelic $\bR$-Cartier divisor such that $D$ is ample, and let $\nu \in \mathcal{V}(X_{\overline{K}})$ be a valuation of maximal rank. Then $\overline{D}$ is ample if and only if $\inf_{\alpha \in \Delta_{\nu}(D)}G_{\overline{D}, \nu}(\alpha) > 0$. 
 \end{coro}

\begin{proof} By Corollary \ref{corominabsG}, the statement is a reformulation of the arithmetic Nakai--Moishezon criterion (Theorem \ref{thmNM}).
\end{proof}

Our next goal is to prove variants of Corollary \ref{coroample} without assuming that $D$ is ample. To do so, we shall  combine  the results of the previous subsection with characterizations of geometric ampleness in terms of Okounkov bodies due to Park and Shin \cite{ParkShin} (which extend the results of Küronya and Lozovanu \cite{KL, KLDoc} to the case of arbitrary characteristic). For this purpose, we focus on valuations arising from admissible and infinitesimal flags constructed in Examples \ref{examplevaladm} and \ref{examplevalinf}.

\subsubsection{Admissible flags} Recall that an admissible flag $Y_\bullet$ on $X_{\overline{K}}$ is a sequence 
\[Y_\bullet\colon X_{\overline{K}} = Y_0 \supsetneq Y_1 \supsetneq \cdots \supsetneq Y_{d-1} \supsetneq Y_{d} = \{x\}\]
such that for every $i \in \{0, \ldots, d\}$, $Y_i \subseteq X_{\overline{K}}$ is a subvariety of codimension $i$, smooth at $x$. We call $x \in X_{\overline{K}}$ the center of $Y_\bullet$. Given an admissible flag $Y_\bullet$ on $X_{\overline{K}}$, we let $\nu_{Y_\bullet} \in \mathcal{V}(X_{\overline{K}})$ be the valuation of maximal rank constructed in Example \ref{examplevaladm}. For any $\overline{D} = (D,(g_v)_{v \in \Sigma_K}) \in \widehat{\Div}(X)_\bR$, we denote by $\widehat{\Delta}_{Y_\bullet}(\overline{D}) = \widehat{\Delta}_{\nu_{Y_\bullet}}(\overline{D})$ the arithmetic Okounkov body defined with respect to $\nu_{Y_\bullet}$. For any real number $\lambda > 0$, we let $\Delta_{\lambda} \subseteq \bR^d_+$ be the standard simplex of size $\lambda$:
\[\Delta_{\lambda} = \{(\alpha_1, \ldots, \alpha_d) \in \bR_+^d \ | \ \sum_{i=1}^d \alpha_i \leq \lambda\}.\]
We have the following arithmetic analogue of \cite[Corollary 3.2]{KLDoc}.
\begin{coro}\label{thmadmample} Assume that $X$ is smooth.  Let $\overline{D} = (D,(g_v)_{v \in \Sigma_K}) \in \widehat{\Div}(X)_\bR$ be a semi-positive adelic $\bR$-Cartier divisor. Then the following assertions are equivalent:
\begin{enumerate}
\item\label{critadmample} $\overline{D}$ is ample;
\item\label{critadmall} there exists a real number $\xi > 0$ with the following property: for every admissible flag $Y_\bullet$ on $X_{\overline{K}}$, there exists $\lambda> 0$ such that $\Delta_{\lambda} \times \{\xi\} \subseteq \widehat{\Delta}_{Y_\bullet}(\overline{D})$;
\item\label{critadmexists} there exists a real number $\xi > 0$ with the following property: for every closed point $x \in X_{\overline{K}}$, there exists an admissible flag $Y_\bullet$ centered at $x$ and a real number $\lambda > 0$ such that $\Delta_{\lambda} \times \{\xi\} \subseteq \widehat{\Delta}_{Y_\bullet}(\overline{D})$.
\end{enumerate}
\end{coro}

\begin{proof} Assume that  $\overline{D}$ is ample  and let $Y_\bullet$ be an admissible flag on $X$. By \cite[Theorem 3.7]{ParkShin}, there exists $\lambda > 0$ such that $\Delta_{\lambda} \subseteq  \Delta_{Y_\bullet}(D):= \Delta_{\nu_{Y_\bullet}}(D)$.  By Corollary \ref{coroample}, $\xi := \inf_{\alpha \in \Delta_{Y_\bullet}(D)} G_{\overline{D},\nu_{Y_\bullet}}(\alpha)> 0$.  Then $\Delta_{\lambda} \times \{\xi\} \subseteq \widehat{\Delta}_{Y_\bullet}(\overline{D})$ by definition of $\widehat{\Delta}_{Y_\bullet}(\overline{D})$. This shows that $\eqref{critadmample} \Rightarrow \eqref{critadmall}$. The implication  $\eqref{critadmall} \Rightarrow \eqref{critadmexists}$ is trivial, so let us prove $\eqref{critadmexists} \Rightarrow \eqref{critadmample}$. Assuming $\eqref{critadmexists}$, we first observe that $D$ is ample by \cite[Theorem 1.1]{ParkShin}, and moreover  $\zeta_{\mathrm{abs}}(\overline{D})\geq \xi > 0$ by Theorem \ref{thmkey}. Therefore $\overline{D}$ is ample by the arithmetic Nakai--Moishezon theorem (Theorem \ref{thmNM}). 
\end{proof}

\subsubsection{Infinitesimal flags}\label{paraginfinitesimal} Let $x \in X_{\overline{K}}$ be a smooth closed point and let $\widetilde{X} \rightarrow X_{\overline{K}}$ be the blow-up of $X_{\overline{K}}$ at $x$, with exceptional divisor $E$. Recall that an infinitesimal flag over $x$ is an admissible flag on $\widetilde{X}$ 
\[\widetilde{Y}_\bullet\colon \ \widetilde{X} = \widetilde{Y}_0 \supsetneq \widetilde{Y}_1 \supsetneq \cdots \supsetneq \widetilde{Y}_{d-1} \supsetneq \widetilde{Y}_{d}\]
such that $\widetilde{Y}_i$ is an $(d-i)$-dimensional linear subspace of $E =  \widetilde{Y}_1 \simeq \bP^{d-1}_{\overline{K}}$ for every $i \in \{1, \ldots, d\}$.  Given an infinitesimal flag over $x \in X_{\overline{K}}$, we let $\nu_{\widetilde{Y}_\bullet} \in \mathcal{V}(X_{\overline{K}})$ be the valuation of maximal rank constructed in Example \ref{examplevalinf}. For any $\overline{D} = (D,(g_v)_{v \in \Sigma_K}) \in \widehat{\Div}(X)_\bR$, we denote by $\widehat{\Delta}_{\widetilde{Y}_\bullet}(\overline{D}) = \widehat{\Delta}_{\nu_{\widetilde{Y}_\bullet}}(\overline{D})$ be the arithmetic Okounkov body defined with respect to $\nu_{\widetilde{Y}_\bullet}$. 
Recall that for $\lambda > 0$, $\Delta_{\lambda}^{-1} \subseteq \bR^d$ denotes the inverted standard simplex of size $\lambda$:
\[\Delta_{\lambda}^{-1} = \{(\alpha_1, \ldots, \alpha_d) \in \bR_+^d \ | \ \alpha_1 \leq \lambda, \ \alpha_2 + \cdots + \alpha_d \leq \alpha_1\}.\]
The following corollary is an arithmetic analogue of \cite[Theorem B]{KL}.
\begin{coro}\label{thminfample}  Assume that $X$ is smooth and let $\overline{D} = (D,(g_v)_{v \in \Sigma_K}) \in \widehat{\Div}(X)_\bR$ be a semi-positive adelic $\bR$-Cartier divisor.   The following assertions are equivalent:
\begin{enumerate}
\item\label{critinfample} $\overline{D}$ is ample;
\item\label{critinfSesh} there exists a real number $\lambda > 0$ such that $ \Delta_{\lambda}^{-1} \ \times \{\lambda\}\subseteq \widehat{\Delta}_{\widetilde{Y}_{\bullet}}(\overline{D})$ for every  infinitesimal flag $\widetilde{Y}_{\bullet}$ on $X_{\overline{K}}$;
\item\label{critinfSeshexists} there exists a real number $\lambda > 0$ such that for every closed point $x \in X_{\overline{K}}$, there exists an infinitesimal flag $\widetilde{Y}_{\bullet}$ over $x$ with $\Delta_{\lambda}^{-1} \times \{\lambda\} \subseteq \widehat{\Delta}_{\widetilde{Y}_{\bullet}}(\overline{D})$.
\item\label{critinfall} there exists a real number $\xi > 0$ with the following property: for every infinitesimal flag $\widetilde{Y}_{\bullet}$ on $X_{\overline{K}}$, there exists a real number $\lambda > 0$ such that $ \Delta_{\lambda}^{-1} \ \times \{\xi\}\subseteq \widehat{\Delta}_{\widetilde{Y}_{\bullet}}(\overline{D})$;
\item\label{critinfexists} there exists a real number $\xi > 0$ with the following property: for every closed point $x \in X_{\overline{K}}$, there exists an infinitesimal flag $\widetilde{Y}_{\bullet}$ over $x$ and a real number $\lambda > 0$ such that $\Delta_{\lambda}^{-1} \times \{\xi\} \subseteq \widehat{\Delta}_{\widetilde{Y}_{\bullet}}(\overline{D})$.
\end{enumerate}

\end{coro}

\begin{proof}
 We first prove $\eqref{critinfample} \Rightarrow \eqref{critinfSesh}$.  Assume that $\overline{D}$ is ample. By Lemma \ref{lemmainfG} and Corollary \ref{coroample}, there exists a real number $\lambda'> 0$ such that $\inf G_{\overline{D}, \nu} \geq \lambda'$ for any valuation of maximal rank $\nu \in \mathcal{V}(X_{\overline{K}})$. 
 For any closed point $x \in X_{\overline{K}}$, we denote by $\epsilon(D,x)$ the Seshadri constant of $D$ at $x$ (see \cite[§~2.3]{ParkShin}). Since $D$ is ample, we have
 \[\epsilon(D) := \inf_{x \in X_{\overline{K}}} \epsilon(D,x)> 0\]
 by Seshadri's criterion for ampleness \cite[Theorem 1.4.13]{LazI}. Let $\lambda$ be a positive real number with $\lambda < \min\{\epsilon(D),\lambda'\}$, and let $\widetilde{Y}_{\bullet}$ be an infinitesimal flag on $X_{\overline{K}}$. By \cite[Theorem 1.2 (2)]{ParkShin}, we have $\Delta_{\lambda}^{-1} \subseteq {\Delta}_{\widetilde{Y}_{\bullet}}(D):= {\Delta}_{\nu_{\widetilde{Y}_{\bullet}}}(D)$. Since $G_{\overline{D},\nu_{\widetilde{Y}_{\bullet}}}(\alpha) \geq \lambda$ for every $\alpha \in {\Delta}_{\widetilde{Y}_{\bullet}}(D)$, we have $ \Delta_{\lambda}^{-1} \ \times \{\lambda\}\subseteq \widehat{\Delta}_{\widetilde{Y}_{\bullet}}(\overline{D})$. The implications $\eqref{critinfSesh} \Rightarrow \eqref{critinfSeshexists} \Rightarrow \eqref{critinfexists}$ and $\eqref{critinfSesh} \Rightarrow \eqref{critinfall} \Rightarrow \eqref{critinfexists}$ are trivial, and the implication $\eqref{critinfexists} \Rightarrow \eqref{critinfample}$ follows from \cite[Theorem 1.1]{ParkShin} and Theorem \ref{thmkey} as in the proof of Corollary \ref{thminfample}.  
\end{proof}

We end this section by comparing our results with the characterizations of arithmetic positivity of toric metrized divisors established by Burgos Gil, Moriwaki, Philippon and Sombra in \cite{BMPS}.

\begin{rema}\label{rematoric} Let $\overline{D} = (D,(g_v)_{v \in \Sigma_K}) \in \widehat{\Div}(X)_\bR$, and assume that the pair $(X,\overline{D})$ is toric in the sense of \cite[Definition 4.12]{BMPS}.  Let $\theta_{\overline{D}}\colon \Delta_D \rightarrow \bR$ be the roof function introduced in \cite[Definition 4.17]{BMPS}. By \cite[§~4.4]{BC}, we can choose an admissible flag $Y_\bullet$ on $X$ such that $\Delta_{Y_\bullet}(D) = \Delta_D$ and $G_{\overline{D},Y_\bullet} = \theta_{\overline{D}}$. Corollary \ref{coroample} gives the equivalence
\[\overline{D} \text{ is ample } \Longleftrightarrow \ D \text{  is ample, } \overline{D} \text{ is semi-positive and }  \inf_{\alpha \in \Delta_D}\theta_{\overline{D}}(\alpha) > 0,\]
which permits to recover \cite[Theorem 2 (1)]{BMPS} thanks to the characterizations of geometric ampleness and semi-positivity provided by \cite[Propositions 4.7 and 4.19]{BMPS}. Similarly, Corollary \ref{coronef} can be considered as a generalization of \cite[Theorem 2 (2)]{BMPS}. In order to fully generalize \cite[Theorem 2]{BMPS} to the non-toric case, it would remain to characterize semi-positivity of adelic $\bR$-Cartier divisors through convex analysis, in analogy with \cite[Proposition 4.19]{BMPS}.
 
\end{rema}
\section{Applications}\label{sectionapplication}

 The goal of this section is to prove Theorems \ref{coroHSintro} and \ref{thmgenericintro}. We shall deduce both of these statements from Corollary \ref{coronef} and from Boucksom and Chen's Theorem \ref{thmBC} relating arithmetic volumes and heights to the arithmetic concave transform.

\subsection{A converse to the arithmetic Hilbert-Samuel theorem}

The following statement is equivalent to Theorem \ref{coroHSintro} in the introduction (see Remark \ref{remacoroHS} below).
\begin{coro}\label{coronefHodge} Let $\overline{D} = (D,(g_v)_{v \in \Sigma_K}) \in \widehat{\Div}(X)_\bR$ be semi-positive, and let $\nu \in \mathcal{V}(X_{\overline{K}})$ be a be a valuation of maximal rank. If $D$ is big, then 
 the following conditions are equivalent:
\begin{enumerate}
\item\label{critHodgenef}  $\overline{D}$ is nef;
\item\label{critHodgevols} $\widehat{\vol}(\overline{D}) =  h_{\overline{D}}(X)$;
\end{enumerate}
\end{coro}

\begin{proof}   We have $ \eqref{critHodgenef} \Rightarrow \eqref{critHodgevols}$ by Theorem \ref{thmHodge}. 
  If  $\widehat{\vol}(\overline{D}) =  h_{\overline{D}}(X)$, then
\[ \int_{\Delta_{\nu}(D)} \max \{0, G_{\overline{D}, \nu}\} d\mu_{\bR^d} = \int_{\Delta_{\nu}(D)} \max \{t, G_{\overline{D},\nu}\} d\mu_{\bR^d} \]
for any real number $t<0$  by Theorem \ref{thmBC}. Therefore $ \mu_{\bR^d}(\{ \alpha \in \Delta_{\nu}(D) \ | \  G_{\overline{D},\nu}(\alpha) < 0\}) = 0$. By upper semi-continuity of $G_{\overline{D},\nu}$, we infer that $G_{\overline{D},\nu}(\alpha) \geq 0$ for any $ \alpha \in \Delta_{\nu}(D)$.  Therefore $\overline{D}$ is nef by Corollary \ref{coronef}. 
\end{proof}

\begin{rema}\label{remacoroHS} Corollary \ref{coronefHodge} is equivalent to Theorem \ref{coroHSintro}. Indeed, for any $\overline{D} = (D,(g_v)_{v \in \Sigma_K}) \in \widehat{\Div}(X)_\bR$ such that $D$ is big and for any $\overline{N} \in \widehat{\Div}(X)_\bR$, we have 
\[\widehat{\vol}(\overline{D}) = \lim_{n\rightarrow \infty} \frac{\widehat{h}^0(X,n\overline{D}+ \overline{N})}{n^{d+1}/(d+1)!}\]
by \cite[Corollary 6.4.10]{CMadelic}. 
\end{rema}

\begin{rema} The assumption that $D$ is big cannot be removed in Corollary \ref{coronefHodge}. To see this, assume that $K$ is a number field and consider the adelic Cartier divisor $\overline{D} = (0, (\xi_v)_{v \in \Sigma_K})$ given by $\xi_v = -2$ if $v$ is archimedean, and $\xi_v = 0$ otherwise. Then $\overline{D}$ is semi-positive, and we have  $\widehat{\vol}(\overline{D}) =  h_{\overline{D}}(X) = 0$. However,  $\widehat{h}_{\overline{D}}(x) = -1$ for every closed point $x \in X_{\overline{K}}$, and therefore $\overline{D}$ is not nef.
\end{rema}

\subsection{Generic nets of small points and subvarieties}

 We now turn to the proof of Theorem \ref{thmgenericintro} (see Theorem \ref{thmgeneric} below). Let us first recall the definition of generic nets of subvarieties.

\begin{defi}\label{defigeneric} We say that a net $(Y_m)_{m \in (\mathcal{I}, \succeq)}$ of subvarieties of $X_{\overline{K}}$ is \emph{generic} if for every Zariski-closed subset $H \varsubsetneq X$, there exists $m_H \in \mathcal{I}$ such that $Y_m \nsubseteq H$ for every $m \succeq m_H$. 
\end{defi}

Given an adelic $\bR$-Cartier divisor $\overline{D} \in \widehat{\Div}(X)_\bR$ on $X$, the essential minimum of $\overline{D}$ is  defined as
\[\zeta_{\mathrm{ess}}(\overline{D}) = \sup_{Z \varsubsetneq X_{\overline{K}}} \inf_{x \in X_{\overline{K}} \setminus Z} \widehat{h}_{\overline{D}}(x),\]
where the supremum is over all the proper Zariski-closed subsets $Z$ of $X_{\overline{K}}$. We have the following variant of Zhang's Theorem on minima \cite[Theorem 5.2]{Zhangplav}.
\begin{theorem}\label{thmZhang}  Let $\overline{D} = (D,(g_v)_{v \in \Sigma_K}) \in \widehat{\Div}(X)_\bR$. If  $D$ is big and  $\overline{D}$ is semi-positive, then
\[\zeta_{\mathrm{ess}}(\overline{D})  \geq \widehat{h}_{\overline{D}}(X) \geq\zeta_{\mathrm{abs}}(\overline{D}).\]  
\end{theorem}

\begin{proof} 
 For the first inequality, see \cite[Theorem 1.5]{BaEM}. The second one is an immediate consequence of Corollary \ref{lemmaminabsinfsubvar}.
\end{proof}

The following theorem gives criteria for equalities to hold in Theorem \ref{thmZhang}.  By Corollary \ref{lemmaminabsinfsubvar}, it implies Theorem \ref{thmgenericintro}.
\begin{theorem}\label{thmgeneric}
 Let $\overline{D} = (D,(g_v)_{v\in \Sigma_K}) \in \widehat{\Div}(X)_\bR$ be semi-positive, with $D$ big. The following conditions are equivalent:
 \begin{enumerate}
 \item\label{itemgenericminabs} $\zeta_{\mathrm{abs}}(\overline{D}) = \widehat{h}_{\overline{D}}(X)$;
 \item\label{itemgenericminess} $\zeta_{\mathrm{ess}}(\overline{D}) = \widehat{h}_{\overline{D}}(X)$;
 \item\label{itemgenericpoints} there exists a generic net of points $(p_m)_{m}$ in $X_{\overline{K}}$ such that 
 \[\lim_m \widehat{h}_{\overline{D}}(p_m) = \widehat{h}_{\overline{D}}(X);\]
 \item\label{itemgenericsubv} there exists a generic net of subvarieties $(Y_m)_{m}$ in $X_{\overline{K}}$ such that 
 \[\lim_m \widehat{h}_{\overline{D}}(Y_m) = \zeta_{\mathrm{abs}}(\overline{D}).\]
  \end{enumerate}
\end{theorem}

\begin{proof} 
 Let $\nu \in \mathcal{V}(X_{\overline{K}})$ be a valuation of maximal rank centered at a regular point of $X_{\overline{K}}$. By Corollary \ref{coronefHodge}  and \cite[Proposition 7.1]{BaEM}, we have 
\begin{equation}\label{eqminG}
\zeta_{\mathrm{abs}}(\overline{D}) = \inf_{\alpha \in \Delta_{\nu}(D)} G_{\overline{D},\nu}(\alpha) \ \text{ and } \ \zeta_{\mathrm{ess}}(\overline{D}) = \max_{\alpha \in \Delta_{\nu}(D)} G_{\overline{D},\nu}(\alpha).
\end{equation}
By Boucksom--Chen's Theorem \ref{thmBC}, we have 
 \[\widehat{h}_{\overline{D}}(X) = \frac{1}{\mu_{\bR^d}(\Delta_{\nu}(D))}\int_{\Delta_{\nu}(D)} \max \{t, G_{\overline{D},\nu} \} d\mu_{\bR^d}\]
for any real number $t \leq \zeta_{\mathrm{abs}}(\overline{D})$. In view of \eqref{eqminG}, it follows that
\[\zeta_{\mathrm{ess}}(\overline{D}) =\widehat{h}_{\overline{D}}(X) \Leftrightarrow G_{\overline{D}, \nu} \text{ is constant} \Leftrightarrow \zeta_{\mathrm{abs}}(\overline{D}) =\widehat{h}_{\overline{D}}(X),\]
that is, $\eqref{itemgenericminabs} \Leftrightarrow \eqref{itemgenericminess}$.

The equivalence $\eqref{itemgenericminess} \Leftrightarrow \eqref{itemgenericpoints}$ is  \cite[Proposition 2.8]{BPRS}. If $\eqref{itemgenericpoints}$ holds, then by the above $\zeta_{\mathrm{abs}}(\overline{D})=\widehat{h}_{\overline{D}}(X)$, and therefore $\eqref{itemgenericsubv}$ holds (for the net $(Y_m)_m = (p_m)_m$). All that remains is to prove the implication $\eqref{itemgenericsubv} \Rightarrow \eqref{itemgenericminabs}$. Assume that there exists a generic net of subvarieties $(Y_m)_{m \in (\mathcal{I},\succeq)}$ in $X_{\overline{K}}$ such that 
 \[\lim_m \widehat{h}_{\overline{D}}(Y_m) = \zeta_{\mathrm{abs}}(\overline{D}),\]
and let $\varepsilon > 0$ be a real number. By \cite[Theorem 1.2]{BaEM}, there exist an integer $n \geq 1$ and a non-zero global section $s \in \Gamma(X,nD)^\times$ such that \[h_{n\overline{D}}(s) = \sum_{v \in \Sigma_{K}} \frac{[K_v :K_{0,v}]}{[K:K_0]} \ln \|s\|_{v,\sup} \leq -n(\zeta_{\mathrm{ess}}(\overline{D}) - \varepsilon) \leq -n(\widehat{h}_{\overline{D}}(X) - \varepsilon) .\] 
 Since the net $(Y_m)_{m\in (\mathcal{I},\succeq)}$ is generic, there exists $m_0 \in \mathcal{I}$ such that $\divi(s)$ intersects $Y_{m_0}$ properly, $\deg_D(Y_{m_0}) > 0$, and 
\begin{equation}\label{eqlimheightnet}
 \zeta_{\mathrm{abs}}(\overline{D}) + \varepsilon \geq \widehat{h}_{\overline{D}}(Y_{m_0}).
\end{equation} 
Let $r = \dim Y_{m_0}$. Then we have
\begin{multline}\label{eqgapheights}
(r+1)\widehat{h}_{\overline{D}}(Y_{m_0}) - r \widehat{h}_{\overline{D}}(\divi(s) \cdot Y_{m_0})\\
= \frac{1}{\deg_D(Y_{m_0})} ({h}_{\overline{D}}(Y_{m_0}) - \frac{1}{n} h_{\overline{D}}(\divi(s) \cdot Y_{m_0})) \geq - \frac{1}{n}h_{n\overline{D}}(s) \geq\widehat{h}_{\overline{D}}(X) - \varepsilon,
\end{multline}
where the penultimate inequality follows from the Bézout formula \cite[(3.13) page 225]{BMPS}.
On the other hand, we have $\widehat{h}_{\overline{D}}(\divi(s) \cdot Y_{m_0}) \geq \zeta_{\mathrm{abs}}(\overline{D})$ by Lemma \ref{lemmaminabsinfsubvar}.
Combining \eqref{eqlimheightnet} and \eqref{eqgapheights}, we have
\[ \zeta_{\mathrm{abs}}(\overline{D}) + (d+1)\varepsilon \geq (r+1)(\zeta_{\mathrm{abs}}(\overline{D}) + \varepsilon) - r\zeta_{\mathrm{abs}}(\overline{D}) \geq \widehat{h}_{\overline{D}}(X) - \varepsilon,\]
and we conclude that \eqref{itemgenericminabs} holds by letting $\varepsilon$ tend to zero. 
\end{proof}

\section*{Acknowledgements} I  thank Mart\'{i}n Sombra for valuable comments and discussions related to this work.

\bibliographystyle{alpha}
\bibliography{okounkovadelic.bib}

\medskip 

\noindent François Ballaÿ\\
 Université Clermont Auvergne, CNRS, LMBP, F-63000 Clermont-Ferrand, France. \\
 \href{mailto:francois.ballay@uca.fr}{\texttt{francois.ballay@uca.fr}}\\
 \href{http://fballay.perso.math.cnrs.fr}{\texttt{fballay.perso.math.cnrs.fr}}

\end{document}